\documentclass[smallheadings,headsepline, 11pt]{amsart}
\usepackage[english]{babel}
\usepackage[all]{xy}
\usepackage{amssymb, amsthm, amsmath,lmodern}
\usepackage{txfonts}
\usepackage{enumitem}
\usepackage{calrsfs}
\usepackage{amscd}
\usepackage{graphicx}
\usepackage[applemac]{inputenc}
\usepackage[T1]{fontenc} 
\usepackage{hyperref} 
\usepackage{amsfonts}
\usepackage{cite}
\usepackage{environ}
\usepackage{xcolor}

\DeclareMathOperator*{\hocolim}{hocolim}

\DeclareMathOperator{\Hot}{H^{0}}

\DeclareMathOperator{\cone}{cone}
\DeclareMathOperator{\MC}{MC}
\DeclareMathOperator{\MCmod}{\mathcal{MC}}
\DeclareMathOperator{\MChmod}{\mathcal{MC}_h}
\DeclareMathOperator{\MCdg}{MC_{dg}}
\DeclareMathOperator{\Tw}{Tw}
\DeclareMathOperator{\Twfg}{Tw_{perf}}
\DeclareMathOperator{\Twfgfree}{Tw_{fg}}

\DeclareMathOperator{\Twm}{{Tw}^m}
\DeclareMathOperator{\Twfgm}{Tw_{perf}^m}
\DeclareMathOperator{\Twred}{Tw^{red}}
\DeclareMathOperator{\Twfgred}{Tw_{perf}^{red}}
\DeclareMathOperator{\ad}{ad}

\DeclareMathOperator{\Hom}{Hom}
\DeclareMathOperator{\uHom}{\underline{Hom}}
\DeclareMathOperator{\sHom}{\cat H \! \it{om}}

\DeclareMathOperator{\End}{End}
\DeclareMathOperator{\uEnd}{\underline{End}}

\DeclareMathOperator{\Ext}{Ext}

\newcommand{\Dperf}{D_{\operatorname{perf}}}
\newcommand{\Dlf}{D_{\operatorname{lf}}}

\input diagxy

\begin{document}
\markright{Maurer-Cartan moduli}
\bibliographystyle{hsiam2}

\setcounter{tocdepth}{2}
\setlength{\parindent} {0pt}
\setlength{\parskip}{1ex plus 0.5ex}

\newcommand{\noproof}{\hfill \ensuremath{\Box}}

\newcommand{\cat}[1]{\mathcal{#1}} 
\newcommand{\ob}{\textrm{ob }}
\newcommand{\mor}{\textrm{mor }}
\newcommand{\id}{\mathbf 1} 
\newcommand{\mods}{\textrm{-Mod}}

\newcommand{\watchit}{\marginpar{$\bigstar$}}

\newcommand{\ground}{k}

\newcommand{\set}[1]{\mathbb{#1}}
\newcommand{\Q}{\mathbb{Q}}
\newcommand{\C}{\mathbb{C}}
\newcommand{\Z}{\mathbb{Z}}
\newcommand{\R}{\mathbb{R}}
\newcommand{\G}{\mathbb{G}}

\newcommand{\PR}{\mathbb{P}}
\newcommand{\A}{\mathcal{A}}
\newcommand{\B}{\mathcal{B}}
\newcommand{\K}{\mathcal{K}}

\newcommand{\De}{\Delta}
\newcommand{\Ga}{\Gamma}
\newcommand{\Om}{\Omega}
\newcommand{\ep}{\epsilon}
\newcommand{\de}{\delta}
\newcommand{\la}{\lambda}
\newcommand{\al}{\alpha}
\newcommand{\om}{\omega}

\renewcommand{\to}{\rightarrow}
\newcommand{\oo}{\infty}
\newcommand{\di}{\mbox{d}}

\newcommand{\op}{^{\textrm{op}}} 

\newcommand{\comment}[1]{}
\newcommand{\margin}[1]{\marginpar{\footnotesize #1}}

\theoremstyle{plain}
\newtheorem{thm}{Theorem}[section]
\newtheorem{cor}[thm]{Corollary}
\newtheorem{lemma}[thm]{Lemma}
\newtheorem{prop}[thm]{Proposition}
\newtheorem{conj}{Conjecture}
\newtheorem{claim}{Claim}

\theoremstyle{definition}
\newtheorem{defn}{Definition}[section]
\newtheorem*{altdef}{Alternative Definition}
\newtheorem{eg}{Example}[section]
\newtheorem*{conv}{Convention}
\newtheorem{fact}{Construction}
\newtheorem{qn}{Question}

\theoremstyle{remark}
\newtheorem{rk}{Remark}[section]

\newtheorem{frk}{Temporary Remark}
\newtheorem{ork}{Temporary Remark}

\def \frk{\color{gray}\begin{ork}}
\def \endfrk{\end{ork}\color{black}}

\NewEnviron{killcontents}{}


\title{Maurer-Cartan moduli and theorems of Riemann-Hilbert type}

\author{Joseph ~Chuang}
\address{Department of Mathematics\\
	City University London\\
	Northampton Square\\
	London EC1V 0HB\\United Kingdom}
\email{j.chuang@city.ac.uk}

\author{Julian Holstein}
\address{Department of Mathematics\\
	Universit\"at Hamburg\\
20146 Hamburg\\
Germany	
}
\email{julian.holstein@uni-hamburg.de}

\author{Andrey~Lazarev}
\address{Department of Mathematics and Statistics\\
	Lancaster University\\
	Lancaster LA1 4YF\\United Kingdom}
\email{a.lazarev@lancaster.ac.uk}
\thanks{This work was partially supported by EPSRC grants EP/N015452/1 and EP/N016505/1}

\date{}
\begin{abstract}
	We study Maurer-Cartan moduli spaces of dg algebras and associated dg categories and show that, while not quasi-isomorphism invariants, they are invariants of strong homotopy type, a natural notion that has not been studied before. We prove, in several different contexts, Schlessinger-Stasheff type theorems comparing the notions of homotopy and gauge equivalence for Maurer-Cartan elements as well as their categorified versions. As an application, we re-prove and generalize Block-Smith's higher Riemann-Hilbert correspondence, and develop its analogue for simplicial complexes and topological spaces.
\end{abstract}
\maketitle
\tableofcontents

\section{Introduction}
The simplest version of the Riemann-Hilbert correspondence is the statement, known for many decades, that the category of flat vector bundles
on a smooth manifold $M$ is equivalent to the category of representations of its fundamental group $\pi_1(M)$. Recently Block and Smith \cite{Block09} developed a higher generalization of this statement. In it, the category of representations of $\pi_1(M)$ was replaced by a differential graded category of infinity local systems on $M$ and the category of flat vector bundles by a differential graded (dg) category of certain modules, called \emph{cohesive modules}, over $\Omega(M)$, the de Rham algebra of $M$. The correspondence was given by a certain $A_\infty$ functor.

The proof in loc.cit.\ is technically complicated and our original motivation was to understand it in simple terms, particularly keeping in mind that one side of the equivalence -- the category of infinity local systems -- is essentially the same as the more classical notion of a cohomologically locally constant (clc) complex of sheaves, i.e.\  a complex of sheaves whose cohomology forms an ordinary (graded) locally constant sheaf. An obvious approach to proving the desired result is based on the observation that $\Omega(M)$ is the global sections of the sheaf of de Rham algebras on $M$ and the latter is a soft resolution of the constant sheaf $\mathbb R$. Similarly, a dg module $N$ over $\Omega(M)$ could be sheafified and viewed as a module over the sheaf of de Rham algebras. Imposing suitable restrictions on $M$, one could hope that the resulting sheaf of modules would be quasi-isomorphic to a clc sheaf and that this procedure establishes an equivalence between the derived category of clc complexes of sheaves on $M$ and a suitable homotopy subcategory of dg $\Omega(M)$-modules (such as cohesive $\Omega(M)$-modules). Taking into account that
the category of clc sheaves makes sense for spaces more general than manifolds, e.g.\ simplicial complexes, one could further ask whether this programme can be carried out in this more general context. Next, one could try to achieve a similar result working with the singular cochain complex of a topological space or a simplicial set, with values in rings other than $\mathbb R$, e.g.\ $\mathbb Z$. Finally, one should study the functorial properties of this construction, in particular its liftability to the suitable homotopy category of spaces that are being considered (manifolds, simplicial complexes, topological spaces or simplicial sets).

Somewhat surprisingly, this naive
approach does work and eventually produces all the results one would initially hope to obtain (and, in fact, quite a bit more). The main difficulty in implementing the strategy outlined above is proving, in different contexts, that the associated complex of sheaves of a dg $\Omega(M)$-module $N$ is clc. 
To show this, one needs to work with Maurer-Cartan (MC) elements in dg algebras and their moduli spaces. 
MC elements and their moduli arise in deformations of various geometric and algebraic objects  (flat connections in vector bundles, complex analytic manifolds \cite{Huybrechts05},  associative algebras \cite{Kontsevich03}), models of function spaces in rational homotopy theory \cite{Lazarev13} and innumerable other contexts of differential and algebraic geometry, homological and homotopical algebra. MC elements are also known as `twisting cochains', particularly in algebro-topological literature \cite{BrownJr59}. 

A priori there are different notions of equivalence for MC elements and it is both necessary for our applications and generally desirable to compare them.
We establish various versions of the classical Schlessinger-Stasheff theorem \cite{Schlessinger12} which states that, under appropriate conditions, homotopy equivalent MC elements must be gauge equivalent, and vice-versa. This result is usually formulated in the context of dg (pro)nilpotent Lie algebras but we need it for dg associative algebras.

Schlessinger-Stasheff type results are established in this paper in two different contexts: analytical (for dg algebras such as the smooth de Rham algebra of a manifold) and algebraic (for dg algebras without any topology or with a pseudo-compact topology such as the singular cochain algebra of a topological space).

The algebraic version of the Schlessinger-Stasheff theorem is particularly interesting and has ramifications far beyond higher Riemann-Hilbert correspondence; some of them have been explored in the present paper but others await further study.

We associate to any dg algebra $A$
 several dg categories, of which the most important is the category of twisted $A$-modules $\Tw(A)$. A version of this category (in the context where $A$ itself is a dg category) was first introduced by Bondal and Kapranov in the seminal paper \cite{Bonda90} where it was called the category of (two-sided) twisted complexes and denoted by $\operatorname{Pre-Tr(A)}$ (in fact, $\Tw(A)$ is obtained from   $\operatorname{Pre-Tr(A)}$ by adding infinite direct sums of objects). The homotopy category $\Hot(\Tw(A))$ is superficially similar to $D(A)$, the derived category of $A$, but is a finer invariant; in particular it is not, generally, a quasi-isomorphism invariant of $A$, unlike $D(A)$ (as pointed out by Drinfeld \cite[Remark 2.6]{Drinfeld04}). It turns out that the correct notion to use in this context is that of \emph{strong homotopy equivalence} of dg algebras. This is a chain homotopy equivalence that takes into account the multiplicative structure and it was not studied before, as far as we know. We show that two strongly homotopy equivalent dg algebras have quasi-equivalent dg categories of twisted modules.

Furthermore, the notion of strong homotopy and strong homotopy equivalence exists also for dg coalgebras (equivalently, pseudo-compact dg algebras), such as the normalized chain complex of a simplicial set, and we show that two weakly equivalent Kan simplicial sets give rise to strongly homotopy equivalent dg coalgebras. This is an important ingredient in the proof of the singular version of the higher Riemann-Hilbert correspondence, but it also has philosophical significance as it shows that the singular chain coalgebra on a simplicial set that is not Kan (or fibrant) might have the wrong homotopy type.
The simple example of a non-fibrant model of the circle $S^1$ shows that this indeed happens, cf. Remark \ref{rem:cobar} below. 
This phenomenon also showed up in the recent paper by Rivera and Zeinalian \cite{Zeina18} where a generalization of Adams' cobar-construction to the non-simply connected case was established.

Denoting by $C^*(X)$ the normalized cochain algebra of a Kan simplicial set $X$, we show that the homotopy category of twisted $C^*(X)$-modules is equivalent to the derived category of clc complexes of sheaves on $|X|$, the geometric realization of $X$. If $X$ is not Kan, the category $\Tw(C^*(X))$ has no homotopy invariant meaning, but one could speculate that it is related to the category of sheaves on $|X|$ that are constructible with respect to some stratification. A related idea is contained in Kontsevich's preprint \cite[pp. 3-4]{Kontsevich09}.

The paper is organized as follows. Section \ref{basic} introduces the notion of an MC element in a dg algebra as well as concomitant concepts: gauge equivalence, MC twisting and a notion of \emph{homotopy gauge equivalence}  
\footnote{It was pointed out to us by the referee that the notion of a homotopy gauge equivalence was already present in \cite{Behrend17} where it was called `quasi-invertible Maurer-Cartan element' and Proposition 8.4 in op.cit. is essentially equivalent to our Theorem \ref{thm-singularss}.} that is, as the name suggests, a relaxation of familiar gauge equivalence to an up to homotopy notion.

Section \ref{twisted} introduces twisted modules, and gives a comparison with Block's cohesive modules \cite{Block10}. In Section \ref{smooth} we study smooth homotopies of topological algebras and their MC elements, and prove an appropriate analogue of the Schlessinger-Stasheff theorem, its categorified version and show that homotopic maps of manifolds give rise to isomorphic functors between the corresponding categories of twisted modules over their de Rham algebras. In Section \ref{strong} we introduce the notions of a strong homotopy of dg algebra morphisms and of a strong homotopy equivalence. A comparison is given with various weaker notions, of which the notion of \emph{derivation homotopy} has been previously known, particularly in the context of rational homotopy theory. We obtain a suitable version of the Schlessinger-Stasheff theorem that implies that strongly homotopy equivalent dg algebras have quasi-equivalent dg categories of twisted modules and obtain a similar result for pseudo-compact dg algebras. In Section \ref{simplicial} we apply our results to normalized cochain algebras of simplicial sets and show that weakly equivalent Kan simplicial sets give rise to quasi-equivalent categories of twisted modules.

In Section \ref{sheaves} we consider complexes of sheaves on a locally ringed space and, using our Schlessinger-Stasheff theorems, show that, under suitable assumptions, the homotopy category of perfect (i.e.\ finitely generated up to homotopy retractions) twisted modules over the dg algebra of global sections is equivalent to the derived category of perfect complexes. This is applied in Section \ref{applications} to produce versions of the higher Riemann-Hilbert correspondence for smooth, possibly non-compact, manifolds and finite-dimensional simplicial complexes, thus generalizing the results of Block and Smith \cite{Block09}. We also consider the case of the Dolbeault algebra and coherent sheaves on a complex manifold, slightly strengthening the result of Block \cite{Block10}. Finally, we treat the most interesting case, that of the singular cochain algebra on a topological space and the corresponding higher Riemann-Hilbert correspondence. The latter is obtained under very general assumptions, i.e.\ we consider any locally contractible topological space and its dg category of possibly infinitely generated and unbounded clc sheaves over any ring of finite homological dimension.

The paper contains an appendix where relevant facts from the theory of nuclear spaces are collected.

\subsection{Notation and conventions}\label{sect-notations}
We work in the category of $\mathbb Z$-graded dg modules over a fixed commutative ring $\ground$; an object in this category is a pair $(V,d_V)$ where $V$ is a graded $\ground$-module and $d_V$ is a differential on it; it will always be assumed to be of cohomological type (so it raises the degree of a homogeneous element). Unmarked tensor products and Homs will be understood to be taken over $\ground$. The  \emph{shift} of a graded $\ground$-module $V$ is the graded $\ground$-module $V[1]$ with $V[1]^i=V^{i+1}$.

A \emph{pseudo-compact relative} graded $\ground$-module is a a projective limit of finitely generated free $\ground$-modules; it is thus complete and separated with respect to the projective limit topology. The adjective `relative' pertains to the discrete ground ring $\ground$; note that in the original definition of Gabriel \cite{Gabriel62} the ground ring is itself supposed to be topological and pseudo-compact modules considered were more general, i.e. not necessarily topologically free. Later on, we shall omit the adjective `relative' as no other pseudocompact modules will be considered.  Pseudo-compact $\ground$-modules form a category where maps are required to be continuous. The category of pseudo-compact $\ground$-modules is anti-equivalent to that of (discrete) free $\ground$-modules via $\ground$-linear
duality. The category of pseudo-compact $\ground$-modules is monoidal:
if $V=\varprojlim V_{\alpha}$ and $U=\varprojlim U_{\beta}$ are two
pseudo-compact $\ground$-modules represented as inverse limits of finitely generated free $\ground$-modules, then $V\hat{\otimes}U:=\varprojlim_{\alpha,\beta}(V_{\alpha}\otimes U_\beta)$. Later on, the hat will always be omitted (but understood) for the tensor product of two pseudo-compact $\ground$-modules. We will also need to form the tensor product of a pseudo-compact $\ground$-module $V=\varprojlim V_{\alpha}$ and a discrete $\ground$-module $U$; such a tensor product will be defined as $V\hat{\otimes} U:=\varprojlim_\alpha(V_{\alpha}\otimes U)$ and, as before, the hat will be omitted but understood. Note that the tensor product of a pseudo-compact and discrete $\ground$-modules has a topology but is not, in general, pseudo-compact. Overviews of this monoidal structure can be found, e.g. in \cite{Hamilton09} (where pseudo-compact modules are called profinite) and in \cite{Vandenberg15}.

A dg algebra is an associative monoid in the dg category of dg $\ground$-modules and in the examples we consider its underlying $\ground$-module is free.  A (right) dg module over a dg algebra $A$ is a dg $\ground$-module $V$ together with a map $V\otimes A\to V$ of dg $\ground$-modules  satisfying the usual conditions of associativity and unitality. Similarly a pseudo-compact dg algebra is a monoid in the monoidal category of pseudo-compact $\ground$-modules. Via continuous linear duality a pseudo-compact dg algebra becomes a dg coalgebra, and the two notions are therefore equivalent. We, however, will work consistently with pseudo-compact algebras rather than coalgebras. An important example of a pseudo-compact dg algebra over $\mathbb Z$ is the singular integer-valued cochain complex $C^*(X,\mathbb Z)$ of a topological space $X$ (or, more pertinently, its normalized version); it is pseudo-compact as dual to the dg coalgebra $C_*(X,\mathbb Z)$ of singular chains on $X$.

We will consider dg \emph{contramodules} over dg pseudo-compact algebras, cf. \cite{Positselski11, Positselski15}; a (right) contramodule over a pseudo-compact algebra $A$ is a \emph{discrete} $\ground$-module $V$ supplied with a `contra-action' map $V\otimes A\to V$ satisfying the usual conditions of associativity and unitality. Note that in loc.cit. a contramodule $M$ over a \emph{coalgebra} $C$ is defined via a structure map $\Hom(C,V)\to V$ satisfying 
suitable conditions; this definition is equivalent, via dualization $A:=C^*$, to ours.

We reiterate that $V\otimes A$ is a \emph{completed} tensor product so a contramodule is not merely an $A$-module where the topology on $A$ is disregarded; at the same time the contra-action map $V\otimes A\to V$ is \emph{not} required to be continuous. Importantly, a contramodule cannot be viewed as a module over a monoid in a symmetric monoidal category in same way as discrete modules or pseudo-compact modules can; this  subtlety  makes the category of contramodules quite peculiar.  Prominent among contramodules are those of the form $V\otimes A$ with the $A$-(contra)action given by the right multiplication. These contramodules are free in the sense that if $U$ is another $A$-contramodule, then $\Hom_A(V\otimes A,U)\cong \Hom_{\ground}(V,U)$ just as it is in the case of usual free $A$-modules. Contramodules encountered in this paper will only be free (and so we will steer clear of various peculiar phenomena alluded to above). For example, $X$ is a topological space and $V$ is a (possibly infinitely generated) free abelian group then $C^*(X,V)\cong V\otimes C^*(X,\mathbb Z)$, the singular cochain complex of $X$ with coefficients in $V$ is a free $C^*(X, \mathbb Z)$-contramodule.

If $M$ is a dg object (such as a dg module, dg algebra etc), we will write $M^{\#}$
for its underlying graded object (i.e. graded module, graded algebra etc).

A \emph{dg category} in this paper will be understood to be a category enriched over dg $\ground$-modules. For example, if $A$ is a dg algebra then the category of dg $A$-modules is a dg category; similarly the category of contramodules over a pseudo-compact dg algebra is also a dg category. The dg $\ground$-module of homomorphisms in a dg category $C$ will be denoted by $\underline{\Hom}(-,-)$ and similarly for endomorphisms. The homotopy category $\Hot(C)$ of the dg category $C$ has the same objects as $C$ and for two objects $O_1, O_2$ in $C$ we have $\Hom_{\Hot(C)}(O_1,O_2):=H^0[\uHom_C(O_1,O_2)]$.

A dg functor $F:C\to C^\prime$ between two dg categories is \emph{quasi-essentially surjective} if $\Hot(F):\Hot(C)\to\Hot(C^\prime)$ is essentially surjective and \emph{quasi-fully faithful} if $F$ induces quasi-isomorphisms on the $\Hom$-spaces; if both conditions are satisfied then $F$ is called a \emph{quasi-equivalence}. A stronger notion is that of a \emph{dg equivalence}: this is a dg functor
$F:C\to C^\prime$  admitting a quasi-inverse dg functor $G:C^\prime\to C$, in the sense that there exist natural closed isomorphisms $F\circ G\cong \id_{C^\prime}$ and $G\circ F\cong\id_{C} $.

A dg category is \emph{strongly pre-triangulated} if it admits cones and shifts, and has a zero object (precise definitions can be found in e.g. \cite{Drinfeld04}), and \emph{pre-triangulated} if it is quasi-equivalent to a strongly pre-triangulated category. A dg functor between  pre-triangulated dg categories is a quasi-equivalence if and only if it induces an equivalence on their homotopy categories.
A category dg-equivalent to a strongly pre-triangulated category is likewise strongly pre-triangulated. Examples of strongly pre-triangulated categories are provided by dg $A$-modules or dg $A$-contramodules where $A$ is a dg algebra or a dg pseudo-compact algebra respectively.

If $X$ is a topological space, we denote by $C_*(X)$ its normalized singular chain dg coalgebra with coefficients in $\ground$ and by $C^*(X)$  its $\ground$-dual normalized cochain (pseudo-compact) dg algebra; similarly if $X$ is a simplicial set, $C_*(X)$ and $C^*(X)$ will stand for its normalized chain dg coalgebra and normalized cochain (pseudo-compact) dg algebra. 

We will call a complex of sheaves on a topological space a \emph{dg sheaf}. For a $\ground$-module $M$ we define by $\underline{M}$ the corresponding constant sheaf on a given topological space. For two dg sheaves $\mathcal {F,G}$ the corresponding dg sheaf of homomorphisms is denoted by ${\sHom}(\mathcal {F,G})$.

We denote by $\Omega(M)$ the de Rham algebra of a smooth manifold $M$. If $K$ is a simplicial complex, then we write $\Omega(K)$ for its \emph{piecewise smooth} de Rham algebra. 
Recall that a smooth form on an $n$-simplex $\Delta^n$ is a smooth form on the interior of $\Delta^n$ such that it and all its derivatives extend continuously to the boundary of $\Delta^n$. It follows from Seeley's extension theorem \cite{Seeley64} that such a form restricts to piecewise smooth forms on the faces of $\Delta^n$. The elements of $\Omega(K)$ are collections of smooth forms on the simplices of $K$ that are compatible with restriction maps. We define the sheaf $\Omega$ on $|K|$, the geometric realization of $K$, by setting $\Omega(U) = \lim_{\Delta^{n}\in K} \cat A^{*}(|\Delta^{n}|\cap U)$ for $U \subset |K|$. Then it is clear that  $\Omega(K)$ coincides with the global sections of $\Om$.
 
When working with complete locally convex spaces $U$ and $V$, we will write $U\otimes V$ for the completed projective tensor product of $U$ and $V$; in the examples relevant to us, $U$ and $V$ will be nuclear, for which this choice of a tensor product is isomorphic to any other reasonable one.  

\subsection{Acknowledgements}
The authors would like to thank Jonathan Block, Chris Braun and Maxim Kontsevich for stimulating discussions, and the anonymous referee for drawing our attention to the paper \cite{Behrend17}.

We also thank Zhaoting Wei for finding some mistakes in an earlier version of this paper and sharing a proof for Lemma \ref{lemma-fineperfect}.

A substantial part of this paper was completed during the third author's visit to IHES, and he wishes to thank this institution for excellent working conditions.

\section{Maurer-Cartan elements for algebras: basic notions, definitions and examples}\label{basic}
Let $A$ be a dg algebra.
\begin{defn}
An element $x\in A^1$ is \emph{Maurer-Cartan}  or MC if it satisfies the equation \begin{equation}\label{eq:MC}  d(x) + x^{2}= 0.\end{equation}
The set of Maurer-Cartan elements in $A$ will be denoted by $\MC(A)$.

The group $A^\times$ of invertible degree 0 elements in $A$ acts on $\MC(A)$ by \emph{gauge equivalences}: for $g\in A^\times, x\in \MC(A)$ set \[g\cdot x=gxg^{-1}-d(g)g^{-1}\]
The \emph{Maurer-Cartan moduli set} $\MCmod(A)$ is the quotient of $\MC(A)$ modulo gauge equivalences.
\end{defn}
We now introduce the notion of MC \emph{twisting}.
\begin{defn}
For $x\in \MC(A)$ the dg $A$ module $A^{[x]}$ has $A$ as its underlying graded space and the differential $d^{[x]}:$
\[
d^{[x]}(a):=d(a)+xa.
\]
The right $A$-module structure on $A^{[x]}$ is the ordinary right multiplication. We will call $A^{[x]}$ the \emph{module twisting} of $A$ by $x$. Similarly define the \emph{algebra} twisting $A^x$ as the dg algebra
having $A$ as an underlying graded algebra and the differential $d^x:$
\[
d^x(a)=d(a)+[x,a].
\]
Note that the MC condition (\ref{eq:MC}) for $x$ implies (in fact, is equivalent to) $d^{[x]}$ squaring to zero in $A^{[x]}$. It also implies that $d^x$ squares to zero in $A^x$. With these definitions, $A^{[x]}$ becomes a dg $(A^x,A)$-bimodule.
\begin{eg}
Let	$X$ be a smooth manifold and $E\to X$ be a flat vector bundle on $X$. Consider $\End(E)$, the associated endomorphism bundle and set $A=\Omega(X,\End(E))$, the de Rham algebra of $X$ with values in $\End(E)$. The given flat structure determines a derivation $d$ on $A$ of square zero; if the bundle $E$ is topologically trivial then $d$ could be taken to be the ordinary de Rham differential. Then an MC element of $A$ is an $\End(E)$-valued 1-form $x$ on $X$ satisfying the MC equation (\ref{eq:MC}). The set $\MC(A)$ is the set of all flat connections on the bundle $E$ and $\MCmod(A)$ is the set of gauge equivalence classes of such flat connections. The complexes $A^{[x]}$ and $A^x$ are respectively one-sided and two-sided twisted de Rham complexes of $X$ with values in $\End(E)$.
\end{eg}
\begin{eg}\label{ex:universal}
Let $A:=\ground[x], d(x)=-x^2$. Clearly $x$ is a non-zero MC element of $A$, and it is not gauge equivalent to 0. This algebra is universal in the sense that an MC element $y$ in a dg algebra $B$ is equivalent to a dg algebra map $A\to B$ with $x\mapsto y$. Note that $A$ is quasi-isomorphic to $\ground$, which implies that the MC moduli set is not quasi-isomorphism invariant.
\end{eg}	
\end{defn}
Recall that  the category of (right) $A$-modules is enriched over dg modules: for any two right dg $A$-modules $M$ and $N$, we have the dg module of homomorphisms  $\underline{\Hom}(M,N)$ from $M$ to $ N$; it is the graded
vector space $\bigoplus_{n=-\infty}^\infty\Hom(M, N[n])$ with the differential $d(f)(m):= df(m)-(-1)^{|f|}f(dm)$. Then we have the following result whose proof is straightforward inspection.
\begin{prop}
Let $x,y\in\MC(A)$.	The dg module $A^{[x,y]}$ of right $A$-module homomorphisms $A^{[x]}\to A^{[y]}$ has $A$ as its underlying graded space and the differential
$d^{[x,y]}:$
\[
d^{[x,y]}(a):=d(a)+ya-(-1)^{|a|}ax.
\]
The operations of left and right multiplications determine a dg $(A^y,A^x)$-bimodule structure on $A^{[x,y]}$. \qed
\end{prop}

Note that for two right $A$-modules $M$ and $N$ a map $M\to N$ of right $A$-modules is precisely a zero-cocycle in $\underline{\Hom}(M,N)$. Then $M$ and $N$ are \emph{homotopy equivalent} if there are maps of (right) $A$-modules $f:M\to N$ and $g:N\to M$ such that $f\circ g$ is cohomologous to $1\in\underline\Hom(N,N)$ and $g\circ f$ is cohomologous to $1\in \underline\Hom(M,M)$.
The notion of a gauge equivalence of MC elements admits an important weakening to a \emph{homotopy gauge equivalence}.
\begin{defn}
Let $\MCdg(A)$ be the dg category whose objects are MC elements of $A$ and for $x,y\in A$
the dg module of morphisms
${\Hom(x,y)_{\MCdg(A)}:=\underline{\Hom}(A^{[x]},A^{[y]})}$. The correspondence $A\mapsto\MCdg(A)$ is a functor from dg algebras to dg categories.

Two MC elements $x,y\in A$ are called \emph{homotopy gauge equivalent} if they are homotopy equivalent as objects in $\MCdg(A)$.
The \emph{Maurer Cartan homotopy moduli set} $\MChmod(A)$ is the set of isomorphism classes of objects in $\Hot(\MCdg(A))$, i.e. the quotient of $\MC(A)$ modulo homotopy gauge equivalences.
\end{defn}
Thus, $x,y\in\MC(A)$ are homotopy gauge equivalent if there exist elements $g, h\in A^0$ such that
\begin{enumerate}
	\item
	$dg+yg-gx=0$;
	\item  $dh+xh-hy=0$;
	\item
	$hg$ is cohomologous to $1$ in $A^x$; 
	\item   $gh$ is cohomologous to $1$ in $A^y$. 
\end{enumerate}
Note that if $g\in A$ is invertible (i.e. $x$ and $y$ are \emph{isomorphic}, as opposed to merely homotopy equivalent in $\MCdg(A)$) then we could take $h=g^{-1}$ and conditions (2), (3) and (4) above are automatically implied by condition
(1). In that case $x$ and $y$ are gauge equivalent. However, the following example shows that the relation of homotopy gauge equivalence is strictly weaker than that of gauge equivalence.
\begin{eg}
Let $A:=\ground\langle x,y,g,h,s,t\rangle$, the free algebra with two generators $x,y$ in degree 1, two generators $g,h$ in degree 0 and two generators $s,t$ in degree -1. The differential in $A$ is given by the formulae:
\begin{align*}	
d(x)&=-x^2,& d(y)&=-y^2,\\
d(g)&=gx-yg, & d(h)&=hy-xh ,\\
d(s)&=-xs+gh-1, & d(t)&=-yt+hg-1.
\end{align*}
It is clear that $x,y\in \MC(A)$ and that $g$ and $h$ provide maps of right dg $A$-modules $A^{[x]}\to A^{[y]}$ and
$A^{[y]}\to A^{[x]}$ respectively that are homotopy equivalences with homotopies given by $s$ and $t$. As an aside, also note that $A$ is the universal dg algebra having two homotopy gauge equivalent MC elements in the sense that any other such dg algebra $B$ receives a unique map from $A$. Now $A$, being free, has no non-scalar invertible elements, and it follows that the MC elements $x$ and $y$ are not gauge equivalent, although they are homotopy gauge equivalent.
\end{eg}

\section{Twisted modules and cohesive modules}\label{twisted}
We will now introduce the notion of a twisted module over a dg algebra $A$.
\begin{defn}
A\emph{	twisted $A$-module} is a (right) dg $A$-module $M$ such that $M^{\#}$ is free as an $A^{\#}$-module. A twisted $A$-module is \emph {finitely generated} if $M^{\#}$ is finitely generated. Finally, any twisted module that is a homotopy retract of a finitely generated  twisted module is called a \emph {perfect twisted module}.

We will denote the dg category of twisted $A$-modules by $\Tw(A)$, and its full subcategories of finitely generated and perfect twisted $A$-modules by $\Twfgfree(A)$ and $\Twfg(A)$ respectively. 
\end{defn}
\begin{rk}
	If $A$ is a dg ring, then a dg $A$-module $M$ is sometimes called perfect if it represents a compact object in the derived category of $A$. This is \emph{not} the same as a perfect twisted $A$-module, in particular the latter need not represent a compact object in a triangulated category. Later on, we will also use the notion of a perfect dg sheaf of modules. In all cases, our terminology will always be clear from the context and unambiguous.
\end{rk}
\begin{rk}
 A twisted $A$-module can be written as $(V \otimes A^{\#}, D_{V})$ where $V$ is a free $\ground$-module and $D_V$ is a  differential on the free $A$-module $V \otimes A^{\#}$ compatible with the $A$-module structure. This is further equivalent to that of an MC element $x\in\End(V)\otimes A$: for such an element $D_V=\id\otimes d_A+x$ gives a differential $D_V$ on $V\otimes A$ compatible with that of $A$ and any compatible differential on $V\otimes A$ must be of this form. We will often slightly abuse notation and write $V \otimes A$ for $(V \otimes A^{\#}, D_{V}).$
\end{rk}

It is easy to see that $\Tw(A), \Twfgfree(A)$ and $\Twfg(A)$ are strongly pre-triangulated dg categories. Shifts are induced by the shift functor on $V$ and the cone on $f: (V \otimes A, D_V) \to (W \otimes A, D_W)$ is given by the complex $(V\oplus W[1])\otimes A$ with differential $\left(\begin{matrix} D_V & f \\  &D_W[1]\end{matrix}\right)$.

The following result shows that the categories $\Tw(A)$ and $\Twfg(A)$ are closed with respect to taking retracts up to homotopy, i.e.\ their homotopy categories are idempotent complete.
\begin{prop}\label{prop:idempotent}
	Any idempotent morphism in $\Hot(\Tw(A))$ or  $\Hot(\Twfg(A))$ is split.
\end{prop}
\begin{proof}
	Since $\Hot(\Tw(A))$ is a triangulated category with direct sums, all idempotents in it split by \cite[Proposition 3.2]{Bokstedt93}. The statement about $\Twfg(A)$ follows directly.
\end{proof}
\begin{rk}\label{remark:Morita}
	We defined the category $\Twfg(A)$ as a certain subcategory of $\Tw(A)$. We see that $\Tw(A)$ is pre-triangulated, with $\Hot(\Tw(A))$ being idempotent complete; it is thus \emph{Morita fibrant}, cf. \cite{Tabuada05} regarding this notion.
	Moreover, the inclusion of the category
	 $\Twfgfree(A)$ of
	finitely generated twisted $A$-modules into $\Twfg(A)$ is a Morita morphism. Thus, $\Twfg(A)$ is a Morita fibrant replacement of $\Twfgfree(A)$ and could be defined, up to a quasi-equivalence, independently of the category $\Tw(A)$.
\end{rk}
The notion of a twisted $A$-module is closely related to that of a \emph{cohesive} $A$-module cf. \cite{Block10}. 
\begin{defn} A right dg $A$-module $M$ is cohesive if $M^{\#}$ is of the form $E \otimes_{A^{0}} A^{\#}$ for a graded right $A^{0}$-module $E$
that is projective, finitely generated in every degree and bounded. 
We will denote the dg category of cohesive $A$-modules by $\cat P_{A}$.\end{defn}
The following result shows that any cohesive $A$-module is, up to a homotopy, a perfect twisted $A$-module.
\begin{prop}\label{prop:retract}
	Any cohesive $A$-module is a retract of a free cohesive $A$-module.
\end{prop}
\begin{proof}
The forgetful functor $A\mods \to A^{\#}\mods$  has a left adjoint sending a (right) $A^{\#}$-module $L$ to the $A$-module $G(L)$ consisting of formal symbols $x+dy$ for $x,y \in L$ with $A$-action given by
\[(x+dy)a=xa+d(ya)-(-1)^{|y|}y\, da \] and differential $d(x+dy) = dx$, see e.g.\ the proof of Theorem 3.6 in \cite{Positselski11}. The unit map $L \to G(L)$ is injective with cokernel isomorphic to $L[-1]$. In particular, if $L$ is projective, then $G(L)^{\#}$ is isomorphic to $L\oplus L[-1]$.
	
Let $P$ be a dg $A$-module and assume that $P^{\#}$ is projective. Let $L$ be a (projective) $A^{\#}$-module such that $P^{\#}\oplus L$ is free. Then $P$ is a retract of $F \coloneqq P \oplus P[-1] \oplus G(L)$ and $F^{\#}$ is isomorphic to $P^{\#}\oplus P^{\#}[-1]\oplus L \oplus L[-1]$, which is a free $A^{\#}$-module. Note that if $P^{\#}$ is finitely generated then $F$ can be chosen so that $F^{\#}$ is of finite rank.
	
In particular a cohesive module $M$ is a retract of a module $F$ such that $F^{\#}$ is a free $A^{\#}$-module of finite rank. But then we can write $F^{\#} = F'\otimes_{A^{0}}A^{\#}$ for some graded $A^{0}$-module $F'$ that is bounded and free of finite rank in every degree, i.e.\ $F$ is a free cohesive module.
\end{proof}

Under mild assumptions cohesive modules and perfect twisted modules agree.

\begin{lemma}\label{lemma-cohesiveretracts}
If $A$ is concentrated in non-negative degrees and $A$ is flat over $A^{0}$ then idempotents split in the homotopy category of cohesive modules.
\end{lemma}
\begin{proof}
We call a bounded complex of finitely generated projective modules over $A^{0}$  \emph{strictly perfect}; any dg-module $A^0$-module quasi-isomorphic to a strictly perfect will be called \emph{perfect}. 

Let $h: E' \otimes_{A^0} A \to E' \otimes_{A^0} A$ be a homotopy idempotent. We can construct a splitting in the homotopy category of all $A$-modules by the well-known telescope trick. Writing $E = E' \otimes_{A^0} A$ we define  a map $\sigma_{h} : \oplus_{\set N} E \rightarrow \oplus_{\set N} E$ defined by sending the $i$-th copy of $E$ to the $(i+1)$-st copy using $h$ and to the $i$-th copy using $1-h$.
Then the cone of $\sigma_{h}$ splits $h$, i.e.\ there is an equivalence $E \simeq \cone(\sigma_{h})\oplus \cone(\sigma_{1-h})$.

By construction $\cone(\sigma_{h})$ is of the form $(N' \otimes_{A^0} A, D_N)$ for some graded  $A^0$-module $N'$. Moreover, inspecting the construction we see that $(N', D_N^{0})$ is equal to $\cone(\sigma_{h^{0}})$, which is the complex obtained by going through the above construction with $(E', h^{0})$ in place of $(E, h)$.
To check this note that the underlying complex of $\cone(\sigma_{h})$ consists of a direct sum of copies of $E' \otimes A$ with some degree shifts. Writing the differential as a matrix each coefficient is given by $\id$, $D_E$ or $h$. Dividing out by $A^{\geq 1}$ leaves a direct sum of shifted copies of $E'$ with differential given by a matrix of $\id$, $D_E^{0}$ and $h^{0}$, which is exactly $\cone(\sigma_{h^{0}})$.

The complex $\cone(\sigma_h)$ is a quasi-cohesive module in the sense of \cite{Block10} (i.e.\ a cohesive module without the assumption of finite generation) and we assumed that $A$ is in non-negative degrees and is flat over $A^{0}$. In this situation Theorem 3.2.7 of loc.\ cit.\ states that $\cone(\sigma_{h})$ is cohesive if $\cone(\sigma_{h^{0}})$ is perfect.

But as $A$ is in non-negative degrees we can check that $h^{0}$ is a homotopy idempotent for $(E', D_E^{0})$ in $A^{0}\mods$. In fact assuming $K$ is a homotopy from $h$ to $h^{2}$ then $K^{0}$ is a homotopy from $h^{0}$ to $(h^{0})^{2}$.
Thus $\cone(\sigma_{h^{0}})$ is a direct summand of $E'$ in the homotopy category.
We claim that this implies it is perfect.
Following \cite{Quillen96} we say a map is algebraically nuclear if it factors through a strictly perfect complex. Then a complex is homotopy equivalent to a strictly perfect complex if and only if the identity is homotopy equivalent to a nuclear map,
see \cite[Proposition 1.1]{Quillen96}.
Since the identity of $\cone(\sigma_{h^{0}})$ factors through $E'$ it is algebraically nuclear up to homotopy. This proves the claim and the lemma.
\end{proof}
\begin{cor}\label{cor-twistedcohesive}
If $A$ is concentrated in non-negative degrees and flat over $A^{0}$ then the dg categories $\Twfg(A)$ and $\cat P_{A}$ are quasi-equivalent.
\end{cor}
\begin{proof}
The inclusion of the dg category of finitely generated twisted modules  $J:\Twfgfree(A) \to \cat P_{A}$  is quasi-fully faithful by construction; moreover it induces, by Proposition \ref{prop:retract} and Lemma \ref{lemma-cohesiveretracts} an equivalence on idempotent completions of its homotopy categories. It follows that $J$ is a Morita morphism and since by Lemma \ref{lemma-cohesiveretracts} $\cat P_{A}$ is Morita fibrant, it could be viewed as a Morita fibrant replacement of $\Twfgfree(A)$. It is, thus, quasi-equivalent to $\Twfg(A)$, cf. Remark \ref{remark:Morita}.
\end{proof}

\begin{eg}
Let $A=\Omega(X)$, the de Rham algebra of a smooth manifold $X$, and $E\to X$ be a flat vector bundle over $X$. Then $\Gamma(X,E)$, the sections of the bundle $E$, form a finitely generated projective module over $A^0$ and the given flat connection form on $E$ determines the structure of a cohesive $A$- module (and thus, of a perfect twisted $A$--module) on  $\Gamma(X,E)\otimes_{A^0} A$.
\end{eg}
The notions described in this and the previous sections make sense when $A$ is a pseudo-compact dg algebra. The definitions of $ \MCmod(A), \MChmod(A), \MCdg(A), \Tw(A)$ and $\Twfg(A)$ are repeated verbatim. One slight subtlety is that the notions of twisted module over $A$ as a pseudo-compact dg algebra and as a discrete dg algebra (i.e. forgetting its pseudo-compact structure) are different, in general. This is because the tensor product of a pseudo-compact algebra and a (discrete) vector space is understood to be completed. A twisted $A$-module in this case is the same as a free $A$-contramodule.
\begin{rk}
In good cases the homotopy category of twisted modules also agrees with Positselski's derived category of the second kind \cite{Positselski11}. It follows from the proof of Proposition \ref{prop:retract} that twisted $A$-modules agree up to homotopy with Positselski's projective $A$-modules $A\mods_{\text {proj}}$. Under certain conditions on $A^{\#}$, the underlying graded algebra of $A$, there is an equivalence $\Hot(A\mods_{\text {proj}}) \cong D^{\text{ctr}}(A\mods)$. See Sections 3.8, 3.9 and 4.4 of \cite{Positselski11}. 
\end{rk}

Thus, for any dg algebra or pseudo-compact dg algebra $A$, we associated several invariants: the dg categories
$\MCdg(A), \Tw(A)$ and $\Twfg(A)$ as well as moduli sets
$\MCmod(A)$ and $\MChmod(A)$.
These are \emph{not} quasi-isomorphism invariants of $A$ as, e.g. Example \ref{ex:universal} demonstrates. Later on we will show that they are, nevertheless, homotopy invariants in two different contexts: analytic and algebraic.

\section{Smooth homotopies for dg algebras}\label{smooth}

In this section we will consider dg Arens-Michael (AM) algebras. These are complete, Hausdorff, locally $m$-convex topological dg algebras over $\mathbb R$. For a detailed introduction see \cite{Pirkovskii08}. A special case of a dg AM algebra is a nuclear dg algebra, e.g the de Rham algebra $\Omega(X)$ where $X$ is a smooth manifold or a simplicial complex.

For our purposes it is enough to know that any dg AM algebra is an inverse limit of dg Banach algebras.
There is a natural notion of smooth homotopy between dg AM algebras.
\begin{defn}
	Let $f_0, f_1:A\to B$ be two continuous maps between dg AM algebras $A$ and $B$. A \emph{smooth homotopy} between $f_0$ and $f_1$ is a map $F:A\to B\otimes\Omega[0,1]$ such that $(\id_A\otimes ev_{0})\circ F=f_0$ and $(\id_A\otimes ev_{1})\circ F=f_1$.
	
	Furthermore, $A$ and $B$ are called smooth homotopy equivalent if there are maps $f:A\to B$ and $g:B\to A$ such that $f\circ g$ and $g\circ f$ are smooth homotopic to $\id_B$ and $\id_A$ respectively.
\end{defn}
\begin{lemma}\label{lem:diagonal1}
	Any AM dg algebra $A$ is smooth homotopy equivalent to $A\otimes\Omega[0,1]$.
\end{lemma}
\begin{proof} It suffices to prove that $\Omega[0,1]$ is smooth homotopy equivalent to $\mathbb R$. This,  in turn, would follow if we show that
	the map $1_{\Omega[0,1]}\circ ev_0:\Omega[0,1]\to \Omega[0,1]$ is smooth homotopic to the identity map on $\Omega[0,1]$. This last homotopy
	can be taken to be the diagonal map $\Delta:\Omega[0,1]\to \Omega[0,1]\otimes \Omega[0,1]$ induced by the multiplication $[0,1]\times[0,1]\to[0,1]$.
\end{proof}
\begin{prop}\label{prop:transitivity}
	The relation of smooth homotopy on morphisms between AM algebras is an equivalence relation.
\end{prop}
\begin{proof}
	Reflexivity is obvious and symmetry follows from the existence of a auto-diffeomorphism of $[0,1]$ swapping the endpoints. For transitivity
	consider a homotopy $F:A\to B\otimes\Omega[0,1]\cong  B\otimes\Omega[0,\frac{1}{2}]$ such that
	 $(\id_B\otimes ev_{0})\circ F=f_1$ and $(\id_B\otimes ev_{\frac{1}{2}})\circ F=f_2$,
	  and another one $G:A\to B\otimes\Omega[0,1]\cong  B\otimes\Omega[\frac{1}{2},1]$ such that 
$(\id_B\otimes ev_{\frac{1}{2}})\circ G=f_2$ and $(\id_B\otimes ev_{1})\circ G=f_3$,
	   The homotopies $F$ and $G$ together constitute a map \[H:A\to B\otimes (\Omega[0,\frac{1}{2}]\times_{\mathbb{R}}\Omega[\frac{1}{2},1])\] where the target of the last map could be viewed as $B$-valued forms on $[0,1]$ that are not necessarily smooth at $\frac{1}{2}$.
	
	To remedy the non-smoothness issue at $\frac{1}{2}$, let $h_1\colon [0,\frac{1}{2}]\to[0,\frac{1}{2}]$ be a smooth function such that $h(0)=0,h(\frac{1}{2})=\frac{1}{2}$ and constant in small neighbourhoods of the endpoints. The correspondence $\omega\to \omega\circ h$ determines a homomorphism
	$\Omega[0,\frac{1}{2}]\to\Omega_0[0,\frac{1}{2}]$ where $\Omega_0$ denotes differential forms constant near the endpoints.  Note that this homomorphism preserves the values of the differential forms at the endpoints. Similarly, there is a homomorphism $h_2:\Omega[\frac{1}{2},1]\to \Omega_0[\frac{1}{2},1]$  preserving the values at endpoints.
 The homomorphisms $h_1$ and $h_2$ together constitute a map 
	\[(\Omega[0,\frac{1}{2}]\times_{\mathbb{R}}\Omega[\frac{1}{2},1]\to
	 (\Omega_0[0,\frac{1}{2}]\times_{\mathbb{R}}\Omega_0[\frac{1}{2},1]\]
	and we denote by $\tilde{h}$ the composition of the latter map with the inclusion $ (\Omega_0[0,\frac{1}{2}]\times_{\mathbb{R}}\Omega_0[\frac{1}{2},1]\subset \Omega[0,1]$; the maps $\tilde{h}$ thus gets rid of a potential non-smoothness at $\frac{1}{2}$. Then 
\[	(\id_B\otimes\tilde{h})\circ H:A\to B\otimes\Omega[0,1]
	\]
	is the desired homotopy between $f_1$ and $f_3$.
\end{proof}
There are also obvious notions of a polynomial or real analytic homotopy, both of which imply smooth homotopy. The relations of polynomial or analytic homotopy are not necessarily transitive.

As in the discrete setting a MC element $x$ in a dg AM algebra $A$ is an element of degree
$1$ such that $dx+x^2=0$ and we can define the gauge action etc.\ in the same way.

\begin{defn}Let $A$ be a dg AM algebra.
	Two MC elements $x_0, x_1\in A$ are called \emph{smoothly homotopic} if there exists a MC element $X\in A\otimes\Omega[0,1]$ such that $(\id_A\otimes ev_{0})(X)=x_0$ and $(\id_A\otimes ev_{1})(X)=x_1$.
\end{defn}
We have the following  result that is a direct consequence of Proposition \ref{prop:transitivity}.
\begin{lemma}
	The relation of smooth homotopy on MC elements of an AM algebra is an equivalence relation.\qed
\end{lemma}\label{lemma:transitive1}

Let $X=x(z)+y(z)dz$ be a smooth homotopy as above. Then it is equivalent to the system of equations
\begin{align}\label{system}
dx(z)+x(z)^2 &= 0, \\
\partial_zx(z)&= -dy(z)+[y(z),x(z)].
\end{align}

\begin{thm}\label{thm-completess}
	Two MC elements $x_0$ and $x_1$ are smoothly homotopic if and only if they are gauge equivalent via an element of $A^{\times}$ in the path component of $1$.
\end{thm}
\begin{proof}
	Note first that we can, without loss of generality, assume that $A$ is a Banach space. Indeed, having a MC element in $A \coloneqq \varprojlim A_\alpha$ where $A_\alpha$ are Banach spaces, is the same as having a compatible collection of MC elements in every $A_\alpha$ (as MC elements are just maps from the algebra $\R[x \ | \ dx + x^{2} = 0]$). The same is true for gauge equivalences and also for homotopies since tensoring with the nuclear space $\Omega[0,1]$ commutes with inverse limits by Theorem \ref{thm:limit}.
	
	The proof is similar to that in \cite[Theorem 4.4]{Chuang10}. Suppose that two MC elements $x_0,x_1\in A$ are gauge equivalent; that means that there exists $g\in A^\times$ for which $g x_0g^{-1}-dg\cdot g^{-1} = x_1$. By assumption, there exists a smooth curve $g(z)$ with $g(0)=1$ and $g(1)=g$. Then define the homotopy $x(z)+y(z)$ in $A$ by $x(z)=g(z)x_0g^{-1}(z)-dg\cdot g^{-1}$ and $y(z)=\partial_z g(z)g^{-1}(z)$. Then a straightforward inspection shows that (\ref{system}) holds.
	
	Conversely, suppose that there is a homotopy $x(z)+y(z)dz$ such that (\ref{system}) holds. Consider the differential equation
	\begin{equation}\label{difeq}
	\partial_z g(z)=y(z)g(z)
	\end{equation}
	with the initial condition $g(0)=1$. (We note that this gives a compatible system of differential equations in Banach algebras.)
	If $g(z)$ is a solution of this differential equation and is invertible in $A$ then  (\ref{system}) would be satisfied with $g(z) x_0g^{-1}(z)-dg\cdot g^{-1}$ in place of $x(z)$ Since a solution of a linear differential equation in a Banach algebra is unique, the solution in the AM algebra $A$ is likewise unique and we conclude that, in fact, $x(z)=g(z)x_0g^{-1}(z)-dg\cdot g^{-1}$ and thus, $x_0$ and $x_1$ are gauge equivalent.
	
	But (\ref{difeq}) does have the solution $g(z)=\operatorname{P}\exp\int_0^zy(t)dt$
	where $\operatorname P\exp$ denotes the path ordered integral, defined by
	\[
	\operatorname{P}\exp\int_0^zy(t)dt \coloneqq 1 + \sum_{n=1}^{\infty} \int_{0\leq t_{1}\leq \dots \leq t_{n}\leq z} y(t_{n}) \cdots y(t_{1}) \ dt_{1}\cdots dt_{n}.
	\]
	By \cite[Propositions 3 and 4]{Araki73}, $g(z)$ is invertible.
\end{proof}
	\begin{cor} \label{cor-completehomotopygauge}
		A smooth homotopy equivalence between two AM dg algebras $A$ and $B$ induces bijections $\MCmod(A)\cong\MCmod(B)$ and $\MChmod(A)\cong\MChmod(B)$.	
	\end{cor}
	\begin{proof}	Let $f: A \to B$ and $g: B \to A$ be homomorphisms such that $g \circ f$ and $f\circ g$ are smoothly homotopic to the identity. Then $f$ and $g$ induce functions between $\MCmod(A)$ and $\MCmod(A)$. To see they are inverse note that, for each $a \in A$, $g \circ f(a)$ is smoothly homotopic to $a$ and thus by Theorem \ref{thm-completess} it is gauge equivalent to $a$, similarly $f \circ g(b)$ is gauge equivalent to $b$. The bijection $\MChmod(A)\cong\MChmod(B)$ is proved in the same way. 
	\end{proof}
We would like to consider twisted modules over AM algebras. Since an endomorphism algebra of an infinite-dimensional space is, in general, not AM, it is not clear whether an arbitrary (infinitely generated) twisted module is a reasonable notion. For an AM algebra $A$, we consider the  dg category $\Twfgfree(A)$ of dg $A$-modules of the form $V\otimes A$ where $V$ is a finite-dimensional $\mathbb R$-space and denote $\Twfg(A)$, its Morita fibrant replacement. The latter can be obtained, e.g. by taking the closure of the Yoneda embedding of $\Twfgfree(A)$ with respect to homotopy idempotents.
\begin{rk}\label{rk:twfgam}
Note that the definition of $\Twfg(A)$ depends, strictly speaking, on whether $A$ is viewed as an AM algebra or a discrete one since in the latter case $\Twfg(A)$ is defined in terms of $\Tw(A)$ which is not considered for an AM algebra $A$. Nevertheless, this is only an ambiguity up to quasi-equivalence since the notion of a Morita fibrant replacement is well-defined up to a quasi-equivalence of dg categories, cf.\ Remark \ref{remark:Morita}. Corollary \ref{cor-twistedcohesive} continues to hold for an AM algebra $A$ with the same proof. 
\end{rk}
A map $A\to B$ of AM algebras induces  functors $\MCdg(A)\to \MCdg(B)$ and $\Twfg(A)\to \Twfg(B)$. It is natural to ask how these induced functors differ for smoothly homotopic maps. The following result answers this question.
\begin{prop}\label{prop:homotopic}
	Let $f,g:A\to B$ be two smoothly homotopic maps of AM algebras. Then the induced functors on \emph{homotopy categories}
	$\Hot\MCdg(A)\to \Hot\MCdg(B)$ and $\Hot\Twfg(A)\to \Hot\Twfg(B)$ are isomorphic.
\end{prop}
\begin{lemma}\label{lem:homotopic}
	For an AM algebra $A$ consider the two natural maps \[\id_A\otimes ev_{0,1}:A\otimes\Omega[0,1]\rightrightarrows A.\] Then the induced functors
	\[\Hot\MCdg(A\otimes\Omega[0,1])\rightrightarrows \Hot\MCdg(A)\] are isomorphic.
\end{lemma}
\begin{proof}
	The map $i:A\to A\otimes \Omega[0,1];a\mapsto a\otimes 1$ induces a quasi-equivalence \[\MCdg(A)(i):\MCdg(A)\to\MCdg(A\otimes\Omega[0,1]).\] 
	Indeed, 
	$i$ is a smooth homotopy equivalence by Lemma \ref{lem:diagonal1}, and
	 Corollary \ref{cor-completehomotopygauge}  then implies that $\MCdg(A)(i)$ is quasi-essentially surjective (even essentially surjective). The quasi-fully faithfulness of $\MCdg(A)(i)$ follows from acyclicity of $\Omega[0,1]$.
	
	The composition $(\id_A\otimes ev_{0})\circ i:A\to A$ is clearly the identity map on $A$ and it follows that the map $\id_A\otimes ev_{0}$ induces the functor $\Hot\MCdg(A\otimes\Omega[0,1])\rightarrow \Hot\MCdg(A)$ that is quasi-inverse to the one induced by $i$. The same can be said about the functor induced by $\id_A\otimes ev_{1}$. Since quasi-inverse functors are determined uniquely up to an isomorphism, the desired claim follows.
\end{proof}
\begin{proof}[Proof of Proposition \ref{prop:homotopic}].
	Let $h:A\to B\otimes\Omega[0,1]$ be a smooth homotopy between $f$ and $g$. Then $(\id_B\otimes ev_1)\circ h=f$ and $(\id_B\otimes ev_2)\circ h=g$, and applying Lemma
	\ref{lem:homotopic} we obtain the desired result.
\end{proof}
\begin{thm}\label{thm:homotopyequivalence1}
	Let $A$ and $B$ be two dg AM algebras that are smoothly homotopy equivalent. Then
	there is a quasi-equivalence between the dg categories
	\begin{enumerate}
		\item $\MCdg(A)$ and $\MCdg(B)$,
		\item $\Twfg(A)$ and $\Twfg(B)$.
	\end{enumerate}
\end{thm}
\begin{proof}
	Let $f:A\to B$ and $g:B\to A$ be the dg algebra maps such that $f\circ g$ is smoothly homotopic
	to $\id_B$ and $g\circ f$ is smoothly homotopic to $\id_A$. These maps
	induce functors $\MCdg(A)(f):\MCdg(A)\to \MCdg(B)$ and
	$\MCdg(A)(g):\MCdg(B)\to\MCdg(A)$. Corollary \ref{cor-completehomotopygauge} implies that $\MCdg(A)(f)$ is quasi-essentially surjective and Proposition \ref{prop:homotopic} -- that $\MCdg(A)(f)$ and $\MCdg(A)(g)$ induce an equivalence $\Hot\MCdg(A)\to \Hot\MCdg(B)$. It follows that $\MCdg(A)$ and $\MCdg(B)$ are quasi-equivalent.
	
	The same argument establishes a quasi-equivalence between $\Twfgfree(A)$ and $\Twfgfree(B)$ after one observes that a  finitely generated twisted $A$-module is the same as an MC element in the dg algebra $A\otimes\End(V)$ where $V$ is a graded finitely generated free $\ground$-module and similarly for $B$. It follows that $\Twfg(A)$ and $\Twfg(B)$ (as Morita fibrant replacements of $\Twfgfree(A)$ and $\Twfgfree(B)$) are quasi-equivalent. 
\end{proof}
If $X$ is a smooth manifold or a simplicial complex we write $\MCdg(X)$ and $\Tw(X)$ for the dg categories  $\MCdg(\Omega(X))$ and $\Tw(\Omega(X))$ respectively. For a smooth (piecewise smooth in the case of simplicial complexes) homotopy $X\times[0,1]\to Y$ of maps between $X$ and $Y$,  the associated map 
$\Omega(Y)\to\Omega(Y\times[0,1])\cong\Omega{Y}\otimes\Omega[0,1]$
 (see Corollary \ref{cor:product} regarding the last isomorphism) is a smooth homotopy of the corresponding dg AM algebras. Therefore, we have the following result.
\begin{cor}\label{cor:deRham}
	Let $f,g:X\to Y$ be (piecewise) smooth homotopic maps between $M$ and $N$. Then the induced functors 
	\[\MCdg(A)(f), \MCdg(A)(g):\Hot\MC(X)\rightrightarrows\Hot\MC(Y);\] 
	\[\MCdg(A)(f), \MCdg(A)(g):\Hot\Tw(X)\rightrightarrows\Hot\Tw(Y)\] are isomorphic.
	
	If $X$ and $Y$ are (piecewise) smooth homotopy equivalent smooth manifolds or simplicial complexes then the following dg categories
	are quasi-equivalent. 
	\[
		\MCdg(X) \text{~and~} \MCdg(Y),
	\]
	\[
	\pushQED{\qed} 
	\Twfg(X) \text{~and~} \Twfg(Y). \qedhere
	\popQED
	\]
\end{cor}

\section{Strong homotopies for dg algebras}\label{strong}
In this section we introduce the notion of strong homotopy between maps of dg algebras and the concomitant notion of strong homotopy equivalence of dg algebras. All definitions, results and proofs are applicable verbatim to dg pseudo-compact algebras as long as we keep in mind our conventions that homomorphisms of pseudo-compact algebras are assumed to be continuous, unmarked tensor products are automatically completed etc.

Let $I$ be the singular simplicial set of the unit interval $[0,1]$; recall that the set $I_n$ of $n$-simplices of $I$ is the set of singular $n$-simplices of $[0,1]$, i.e. the set of continuous maps $\Delta^n\to[0,1]$ where $\Delta^n$ is the standard topological $n$-simplex. We will consider a collection of simplicial subsets of $I$ defined as follows.

\begin{enumerate}
\item The simplicial set $K_0$ is generated by two nondegenerate simplices $a_0,b_0$ in degree zero corresponding to the endpoints of $[0,1]$ and one nondegenerate simplex $a_1$ in degree one corresponding to the linear path from $0$ to $1$ in $[0,1]$, viewed as a $1$-simplex in $[0,1]$.
\item The simplicial set $K_1$ contains all the simplices of $K_0$ and has, additionally, one other nondegenerate $1$-simplex $b_1$ corresponding to the linear path from $1$ to $0$ in $[0,1]$.
\item Assuming that for $n\geq 1$ the simplicial set $K_n$ has been defined, we let $K_{n+1}$ contain all the simplices of $K_n$ plus two additional nondegenerate simplices $a_n,b_n$ in degree $n$ defined as follows. Writing $\Delta^n$ as the convex hull of its vertices $x_0,\ldots, x_n$, we let $a_n:\Delta^n\to[0,1]$ and $b_n:\Delta^n\to[0,1]$ be the affine maps for which
\[a(x_i)=\begin{cases}0, \text{if $i$ is even,}\\1,\text{if $i$ is odd}\end{cases} \text{and~} b(x_i)=\begin{cases}1, \text{if $i$ is even,}\\0,\text{if $i$ is odd.}\end{cases}\]
\item The simplicial set $K_\infty$ is the union of the nested sequence of simplicial sets $K_1\subset K_2\subset\ldots$.
\end{enumerate}
\begin{rk}
We have the following inclusions of the simplicial sets introduced above:
\[K_0\subset K_1\subset\ldots \subset K_\infty\subset I\]
as well as their geometric realizations $|K_n|$. It is clear that $|K_0|$ is a cell decomposition of $[0,1]$ with two 0-cells and one 1-cell. Furthermore, for $n=1,\ldots,\infty$ the cell complex $|K_n|$ is homeomorphic to the $n$-sphere $S^n$ with two cells in each dimension.
\end{rk}
\begin{lemma}\label{lemma:retract}
	The simplicial set $K_\infty$ is a retract of $I$.
\end{lemma}
\begin{proof}
Consider the category $\mathcal K$ with two objects and two mutually inverse morphisms between them. The simplicial set $K_\infty$ is, by definition, the nerve of $\mathcal K$. Since $\mathcal K$ is a groupoid, its nerve is a Kan simplicial set (cf. \cite[Lemma 3.5]{Goerss99}). Clearly, the inclusion $K_\infty\to I$ is an acyclic cofibration and it follows that it admits a splitting, exhibiting $K_\infty$ as a retract of $I$.
\end{proof}
\begin{rk}
Since the simplicial sets $K_n, 1\leq n<\infty$ are not contractible, they are not retracts of $I$. The simplicial set $K_0$, while contractible, is still not a retract of $I$ since it is not Kan.
\end{rk}	
We denote by $K_n^*, n=0,1,\ldots, \infty$ and $I^*$ the complexes of normalized cochains on the corresponding simplicial sets with values in $\ground$.  Endowed with the Alexander-Whitney product, these become dg algebras, in fact pseudo-compact dg algebras (as duals to dg coalgebras). We re-iterate that, even though $K_\infty^*$ is a degree-wise finitely generated free $\ground$-module, it will be regarded as pseudo-compact, in particular tensor products with it will always be understood in the completed sense, as per our convention. Note that this subtlety is vacuous for $K_n^*, n<\infty$ as these free $\ground$-modules have totally finite rank. We have the following tower of surjective maps of dg pseudo-compact algebras:
\[
K_0^*\leftarrow K_1^*\leftarrow\ldots\leftarrow K_\infty^*\leftarrow I^*.
\]
Note that any pseudo-compact dg algebra in this tower admits two maps $ev_1$ and $ev_2$ to $\ground$ corresponding to the inclusion of the two endpoints of $[0,1]$ into the corresponding simplicial set.  We can define the notion of a $K$-multiplicative homotopy of dg algebra maps where $K$ is any simplicial subset of $I$ containing the 0-simplices corresponding to the endpoints of $[0,1]$. In the following definition $K$ is $K_n, n=0, 1,\ldots,\infty$, or $I$.
\begin{defn}
Let $f,g:A\to B$ be two dg algebra maps. An \emph{elementary $K$-homotopy} between them is a map $H:A\to B\otimes K^*$ such that $(\id_B \otimes ev_1)(H)=f$ and  $(\id_B\otimes ev_2)(H)=g$. We say that $f$ and $g$ are $K$-homotopic, if they are related by the equivalence relation generated by elementary $K$-homotopy. If $K=I$, we will refer to $K$-homotopy as \emph{strong homotopy}.

Furthermore, $A$ and $B$ are called $K$ homotopy equivalent if there are maps $f:A\to B$ and $g:B\to A$ such that $f\circ g$ and $g\circ f$ are $K$-homotopic to $\id_B$ and $\id_A$ respectively. If $K=I$ we will refer to a $K$ homotopy equivalence as a \emph{strong homotopy equivalence}.
\end{defn}
\begin{rk}
It is easy to see that for $n>0$ the relation of elementary homotopy is symmetric but not transitive and for $n=0$ it is not even symmetric. Furthermore, using \emph{normalized} cochains is essential: almost all of our results will fail for un-normalized cochains. For example,
the un-normalized singular cochain algebra of the one-point topological space is the dg algebra of Example \ref{ex:universal} having non-trivial
dg categories of MC elements and twisted modules.
	
Since $K_\infty$ is a retract of $I$, the notions of strong homotopy and strong homotopy equivalence are equivalent to those of a $K_\infty$	homotopy and $K_\infty$ homotopy equivalence respectively. It is this notion of multiplicative homotopy that is of chief relevance to this paper. Also of interest is  the notion of $K_0$ homotopy
(sometimes called \emph{derivation homotopy}); it has been used in rational homotopy theory, cf. for example \cite{Anick89}.
\end{rk}
\begin{lemma}\label{lem:diagonal2}
Any dg algebra $A$ is strongly homotopy equivalent to $A\otimes I^*$ (and thus, also to $A\otimes K_\infty^*$).	
\end{lemma}	
\begin{proof}
The multiplication map $[0,1]\times[0,1]\to[0,1]$ makes $I^*$ into a bialgebra and the coproduct map $I^*\to I^*\otimes I^*$ could be viewed as
a strong homotopy between the identity map on $I^*$ and a projection onto $\ground$. It follows that $I^*$ (and thus, $K_\infty^*$) is strongly	homotopy equivalent to $\ground$ and the desired statement is an immediate consequence.
\end{proof}
\begin{rk}  Similarly, $K_0^*$ is $K_0$-homotopy equivalent to $\ground$ and so $A$ is $K_0$ homotopy equivalent to $A\otimes K_0$; we will not use this result. Since for $0<n<\infty$ the algebra $K_n^*$ is not acyclic, it is not $K_n^*$ homotopy equivalent to $\ground$.
We will see later on (Example \ref{eg:K_0}) that $K_0^*$ is not $K_2$ homotopy equivalent to $\ground$.
\end{rk}
\begin{prop}\label{prop:comp}
	Let $f,g:A\to B$ be two $K$-homotopic dg algebra maps. If $C$ is a third dg algebra then for any dg algebra map $h:C\to A$ the maps $f\circ h,g\circ h:C\to B$ are $K$-homotopic. Similarly for any dg algebra map $k:B\to C$ the maps $k\circ f, k\circ g:A\to C$ are $K$-homotopic.
\end{prop}
\begin{proof}
It suffices to treat the case of an elementary homotopy.	If $H:A\to B\otimes K^*$ is an elementary homotopy between $f$ and $g$ then $H\circ h$ is an (elementary) $K$-homotopy between $f\circ h$ and $g\circ h$. Similarly,
	$(k\otimes \id_{K^*})\circ H$ is an (elementary) homotopy between $k\circ f$ and $ k\circ g$.
\end{proof}
\begin{rk}
	Using Proposition \ref{prop:comp} we can define the $K$-homotopy category of dg algebras as having dg algebras as objects and $K$-homotopy classes of maps as morphisms. Of most interest is the case $K=K_\infty$ and $K=K_0$ as $K_\infty^*$ and $K_0^*$ are acyclic dg algebras. As was mentioned earlier, $K_0^*$ is not $K_\infty$ contractible and so, the relation of $K_\infty$ homotopy equivalence is strictly finer than that of $K_0$-equivalence.
	
	Moreover, the existence of a $K$-homotopy category of dg algebras suggest the existence of a closed model category structure underpinning it.
	The standard closed model structures on dg algebras having quasi-isomorphisms as weak equivalences, should then be localizations of the $K_\infty$ closed model structure.
\end{rk}	
The main advantage of the dg algebra $K^*_\infty$ over $I^*$ is that the former is much smaller, and $K_{\oo}$, as well as its quotients $K_n,n<\infty$, admits an explicit description.
\begin{prop}\begin{enumerate}
		\item
The dg algebra $K_\infty^*$ is generated by two elements $e,f$ in degree zero and two elements $s,t$ in degree one, subject to the relations
\begin{align*}
e^2 &=e, &f^2 &=f, & ef &=fe=0;\\
fs&=s,&se&=s, &sf &=es=0;\\
tf&=t, &et&=t, &ft &=te=0;\\
t^2 &=s^2=0
\end{align*}
with the differential specified by the formulae
\begin{align*}
d(e)&=t-s, &d(f)&=s-t; \\
d(s)&=ts+st,
&d(t)&=st+ts.
\end{align*}
\item The algebra $K_n^*,0<n<\infty$ is the quotient of $K_\infty^*$ by the dg ideal generated by monomials in $s$ and $t$ of length $>n$.
\item The algebra $K_0^*$ is the quotient of $K_\infty^*$ by the dg ideal generated by $t$ and polynomials in $s$ of degree $>1$.
\end{enumerate}
\end{prop}
\begin{proof}
Statements (2) and (3) clearly follow from (1). To prove (1), we use the interpretation of $K^*_\infty$ as the normalized cochain algebra of the
 nerve of the category with two objects and two mutually inverse morphisms between them as in the proof of Lemma \ref{lemma:retract}. It follows that $K^*_\infty$ is the path algebra of the graded quiver
 \xymatrix
 {
 	\bullet\ar@/_/[r]^s&\bullet\ar@/_/[l]_t
} with arrows $s$ and $t$ placed in degree 1. The stated relations in $K^*$ are precisely the relations in this path algebra, with the elements $e$ and $f$ corresponding to the length zero paths at the vertices of the above quiver. The formula for the differential in $K^*_\infty$ is straightforward to obtain.
\end{proof}
\subsection{Strong homotopies for MC elements}
There is a corresponding notion of K-homotopy for MC elements.
\begin{defn}
Two MC elements $x_0, x_1$ in a dg algebra $A$ are called K-homotopic if there exists an MC element $X\in A \otimes K^*$ such that $(\id_A \otimes ev_0)(X)=x_0$ and $(\id_A \otimes ev_1)(X)=x_1$. If $K=I$, this will be referred to as \emph{strong homotopy} of MC elements.
\end{defn}

It turns out that the notions of $K_2$ homotopy and homotopy gauge equivalence for MC elements are equivalent.
\begin{lemma}\label{lemma:K2gauge}
	Let $x,x^\prime$ be two MC elements in a dg algebra $A$. Then $x$ and $x^\prime$ are homotopy gauge equivalent if and only if they are
	$K_2$-homotopic.
\end{lemma}
\begin{proof} Let $X\in A\otimes K^*_2$ be a $K_2$-homotopy between $x$ and $x^\prime$. We could write
	\[
X=xe+x^\prime f+ys+y^\prime t+zts+z^\prime st,	
	\]
where $y,y^\prime$ and $z,z^\prime$ are elements of $A$ of degrees $0$ and $1$ respectively. Writing down the MC equation for $X$ and equating to zero the coefficients at $e,f,s,t, st$ and $ts$ we obtain:
\begin{align*}
x^2+dx=0;
(x^\prime)^2+dx^\prime &=0\\
dy+x^\prime(y+1)-(y+1)x &=0\\
dy^\prime+x(y^\prime+1)-(y^\prime+1)x^\prime &=0\\
(y+1)(y^\prime+1)-1+dz+[x,z] &=0\\
(y^\prime+1)(y+1)-1+dz^\prime+[x^\prime,z^\prime] &=0.
\end{align*}	
The first line above is the statement that $x$ and $x^\prime$ are MC elements in $A$, the second and third -- that the elements $y+1$
and $y^\prime+1$ determine right $A$-module maps
$A^{[x]}\to A^{[x^\prime]}$ and $A^{[x^\prime]}\to A^{[x]}$
respectively and the last two lines -- that the elements $(y+1)(y^\prime+1)$ and $(y^\prime+1)(y+1)$ are cohomologous to $1$ in $A^x$ and $A^{x^\prime}$ respectively. It follows that $x$ and $x^\prime$ are homotopy gauge equivalent. Conversely, if $x$ and $x^\prime$ are homotopy gauge equivalent, then performing the above calculations in the reverse order, we find a $K_2$-homotopy between $x$ and $x^\prime$.
\end{proof}
Rather surprisingly, the notions of $K_2$ and $K_\infty$ homotopy for MC elements are equivalent. This could be interpreted as a strong homotopy analogue of the Schlessinger-Stasheff theorem. Strikingly, it holds with no assumptions on the dg algebra in question. To show this, we need a few preliminary results. Recall that we introduced a category $\mathcal K$ having two objects $O_1$ and $O_2$ and two mutually inverse morphisms between them. Let $\mathcal K_\infty$ be the dg category  with the same set of objects $O_1$ and $O_2$ and a set of free generators:
\[
x_n:O_1\to O_2; y_n:O_2\to O_1 \text{ for } n=0,1,\ldots
\]
with $|x_n|=|y_n|=n$. The differential $d$ is given on the generators as follows:
\begin{align*}
d(x_0)&=0, &d(y_0)&=0;\\
d(x_1)&=y_0x_0-1, &d(y_1)&=x_0y_0-1
\end{align*} and for $n>0$:
\begin{eqnarray*}
d(x_{2n})=& \sum_{i=0}^{n-1}(x_{2i}x_{2(m-i)-1}-y_{2(m-i)-1}y_{2i});\\
d(y_{2n})=& \sum_{i=0}^{n-1}(y_{2i}y_{2(n-i)-1}-x_{2(n-i)-1}y_{2i});\\
d(x_{2n+1})=&\sum_{i=0}^ny_{2i}x_{2(n-i)}-\sum_{i=0}^{n-1}x_{2i-1}x_{2(n-i)-1};\\
d(y_{2n+1})=&\sum_{i=0}^nx_{2i}y_{2(n-i)}-\sum_{i=0}^{n-1}y_{2i+1}y_{2(n-i)-1}.
\end{eqnarray*}
Note that $\mathcal K_\infty$ is a cofibrant dg category. Clearly there is a surjection $\mathcal K_\infty\to \mathcal K$ whose kernel is generated by all $x_n, y_n,n>0$. Then we have the following result.
\begin{lemma}\label{lemma:resolution}
	The map $\mathcal K_\infty\to \mathcal K$ is a quasi-isomorphism, i.e. $\mathcal K_\infty$ is a cofibrant resolution of $\mathcal K$.
\end{lemma}
\begin{proof}
	This is proved in \cite[Theorem 9,]{Markl01}; note that $\mathcal K$ and $\mathcal K_\infty$ are called `coloured operads' in the cited reference but these are really dg categories as they do not support operations of higher arities.
\end{proof}	
\begin{rk}\
	\begin{itemize}
\item
The proof of Lemma \ref{lemma:resolution} in \cite{Markl01} is computational. In fact, the resolution $\mathcal K_\infty\to \mathcal K$ is the standard reduced bar-cobar resolution of the category $\mathcal K$. The existence of such a resolution seems to be a well-known fact and is mentioned, in, e.g. \cite{Drinfeld04, Keller06}. We are, however, unaware of any reference where this general fact has been given a full proof.
\item A different (smaller) resolution of the category $\mathcal K$ was described in \cite[Corollary 3.7.3]{Drinfeld04}.
\item A one object analogue of the dg category $\mathcal K$ is the algebra $\ground[s,s^{-1}]$ with $|s|=0$. A cofibrant resolution of this algebra was constructed in \cite{BCL18}; the formulae are essentially the same as for $\mathcal K_\infty$.
	\end{itemize}
\end{rk}
\begin{lemma}\label{lemma:category-algebra}
	Let $A$ be a dg algebra and $x, x^\prime\in \MC(A)$. Then there is a 1-1 correspondence between strong homotopies from $x$ to $x^\prime$ and
	dg functors $\mathcal K_\infty\to\MCdg(A)$ sending $O_1$ and $O_2$ to $x$ and $x^\prime$ respectively.
\end{lemma}	
\begin{proof}
	Let $X\in \MC(A\otimes K^*_\infty)$ be a strong homotopy from $x$ to $x^\prime$. We could write
	\[
	X=xe+x^\prime f +u_0s+v_0t+u_1st+v_1ts+\ldots.
	\]
	In other words the coefficient of $X$ at the monomial $st\ldots t$ or $st\ldots s$ of length $n$ is $u_n$ and the coefficient
	at the monomial $ts\ldots t$ or $ts\ldots s$ of length $n$ is $v_n$. Note that the $u_n,v_n$ are elements of $A$ of degree $n$.
	
	Similarly, a dg functor $F:\mathcal C\to\MCdg(A)$ such that $F(O_1)=x$ and $F(O_2)=x^\prime$ is determined (since $\mathcal C$ is a free category) by a collection of elements
	\begin{align*}
	F(x_{2n}) &\in \Hom_{\MCdg(A)}(x,x^\prime),\\
	F(y_{2n}) &\in \Hom_{\MCdg(A)}(x^\prime,x),\\
	F(x_{2n+1})&\in \Hom_{\MCdg(A)}(x,x) ,\\
	F(y_{2n+1})&\in \Hom_{\MCdg(A)}(x^{\prime},x^\prime)
	\end{align*}
	where $n=0,1,\ldots$.
	
	The correspondence between these two sets of data is given by \[F(x_0)=u_0+1, F(y_0)=v_0+1\] and, for $n>0$:
	 \[F(x_n)=u_n, F(y_n)=v_n.\]
	 Finally, a somewhat tedious but straightforward calculation, similar to that of Lemma \ref{lemma:K2gauge} shows that the MC equation $d(X)+X^2=0$
	 translates into the condition that $F$ is a dg functor (i.e. determines a dg map on $\Hom$-complexes).
\end{proof}
\begin{thm}\label{thm-singularss}
Let $A$ be a dg algebra. Then two MC elements in $A$ are strongly homotopic if and only if they are homotopy gauge equivalent.
\end{thm}
\begin{proof}
If two MC elements  $x, x^\prime$ in $A$ are strongly (or $K_\infty$) homotopic then they are $K_2$ homotopic since $K_2^*$ is a quotient of $K_\infty^*$ and thus by Lemma \ref{lemma:K2gauge} they are homotopy gauge equivalent.

Conversely, let $x,x^\prime\in \MC(A)$ be homotopy gauge equivalent and consider a map $f:x\to x^\prime$ inducing an isomorphism in $\Hot(\MCdg(A))$.
Let $\mathcal K_0$ be the $\ground$-linear category generated by two objects $O$ and $O^\prime$ and a single morphism $i:O\to O^\prime$. Then there is
a unique dg functor $F:\mathcal K_0\to \MCdg(A)$ mapping $i$ to $f$. Since $f$ represents an isomorphism in $\Hot(\MCdg(A))$, the functor $F$ factors through $L_i(\mathcal K_0)$, the derived localisation of $\mathcal K_0$, cf. \cite{Toen07b}. On the other hand, it follows from the proof of
\cite[Corollary 9.7]{Toen07b} that $L_i(\mathcal K_0)$ is quasi-equivalent to the category $\mathcal K$ consisting of two mutually inverse isomorphisms between two objects $O_1$ and $O_2$. Since $\mathcal K_\infty$ is a cofibrant replacement of $\mathcal K$, we obtain a dg functor
$\mathcal K_\infty\to \MCdg(A)$ taking $O_1$ and $O_1$ to $x$ and $x^\prime$ respectively. By Lemma \ref{lemma:category-algebra} this implies that $x$ and $x^\prime$ are strongly homotopic.
\end{proof}
\begin{cor}
	For $n=2,3\ldots,\infty$ the relation of $K_n$-homotopy on MC elements of a dg algebra is an equivalence relation.
\end{cor}
\begin{proof}
	Indeed, by Theorem \ref{thm-singularss} two MC elements in a dg algebra $A$ are $K_n$-homotopic  if and only if they are homotopy equivalent as objects in $\MCdg(A)$. The latter relation is obviously an equivalence relation.
\end{proof}
\begin{cor}\label{cor:strongh}
	A strong homotopy equivalence between two dg algebras $A$
	and $B$ induces a bijection $\MChmod(A)\cong\MChmod(B)$.
\end{cor}
\begin{proof}
	The given strong homotopy equivalence between A and B clearly induces a bijection of MC elements up to strong homotopy. By Theorem \ref{thm-singularss} this becomes a bijection of MC elements up to homotopy gauge equivalence, i.e. a bijection $\MChmod(A)\cong\MChmod(B)$.
\end{proof}
The following result is a strong homotopy analogue of Proposition \ref{prop:homotopic} and its proof is completely analogous, after  replacing $\Omega[0,1]$ with $I^*$ and Corollary \ref{cor-completehomotopygauge} with Corollary \ref{cor:strongh}.
\begin{prop}\label{prop:homotopic1}
Let $f,g : A\to B$ be two strongly homotopic maps of dg algebras. Then the induced functors on homotopy categories: \begin{enumerate}\item
$\Hot\MCdg(f), \Hot\MCdg(g):	\Hot\MCdg(A)\to \Hot\MCdg(B)$;\item
$\Hot\Tw(f), \Hot\Tw(g):	\Hot\Tw(A)\to \Hot\Tw(B)$;\item	
$\Hot\Twfg(f), \Hot\Twfg(g):	\Hot\Twfg(A)\to \Hot\Twfg(B)$
\end{enumerate}	
 are isomorphic.\qed
\end{prop}	
The following result is a strong homotopy analogue of Theorem \ref{thm:homotopyequivalence1} and the proof is completely analogous, after replacing $\Omega[0,1]$ with $I^*$, Corollary \ref{cor-completehomotopygauge} with Corollary \ref{cor:strongh} and Proposition \ref{prop:homotopic} with Proposition \ref{prop:homotopic1}.
\begin{thm}\label{thm:homotopyequivalence}
	Let $A$ and $B$ be two dg algebras that are strongly homotopy equivalent. Then
	there are quasi-equivalences of dg categories between
		\begin{enumerate}
			\item $\MCdg(A)$ and $\MCdg(B)$,
			\item $\Tw(A)$ and $\Tw(B)$,
			\item $\Twfg(A)$ and $\Twfg(B)$.\qed
		\end{enumerate}
\end{thm}

\begin{eg}\label{eg:K_0}
	Assume that $2$ is not invertible in $\ground$ and consider the dg algebra $K_0^*$; recall that it is the path algebra of the quiver \xymatrix
	{
		\bullet\ar[r]^s&\bullet
	}
with $|s|=1$ with the differential being $\ad(s)$. It is clear that $s\in K^*_0$ is an MC element. It is easy to see that the $K^*_0$-module
$K^{*[x]}_0$ is not isomorphic to $K^*_0$ as its first homology group is $\ground/2\neq 0$. It follows that $s$ is not homotopy gauge equivalent to zero and therefore by Theorem \ref{thm:homotopyequivalence}, $K^*_0$ is not $K_2$ homotopy equivalent to $\ground$.
\end{eg}
\begin{eg} \label{eg-polyderham}
Now let $\ground$ be a field of characteristic zero and consider $A:=\ground[z,dz]$, the polynomial de Rham algebra of the line.
This is quasi-isomorphic to $\ground$, and even polynomially homotopy equivalent to $\ground$, however $\MCdg(A)$ is not quasi-equivalent to $\MCdg(\ground)$ (and so, $A$ is not $K_2$ homotopy equivalent to $\ground$). To see this note that MC elements are just polynomial 1-forms and a map in $\MCdg(\ground[z,dz])$ between two such elements $x$ and $y$ is a polynomial $f$ satisfying $df + fx + yf = 0$. This differential equation will not usually have polynomial solutions, so different choices of $x$ and $y$ give a large number of MC elements in $A$ which do not map to one another (and thus represent non-isomorphic objects in $\Hot\MCdg(A)$).
	The finitely generated twisted modules represented by these $\MC$ elements are examples of $\cat O$-coherent $\cat D$-modules with irregular singularities at infinity.
\end{eg}
\begin{rk}
	It is interesting to find out whether there is a closed model category on dg algebras with weak equivalences being being what we call strong homotopy equivalences. Such a closed model category cannot be transferred from the category of complexes. For example, if $\ground$ is a field of characteristic zero, the de Rham algebra $\ground[z,dz]$ is chain homotopy equivalent to $\ground$ as a complex of $\ground$ vector spaces, but supports many nontrivial  MC element and so, cannot be strongly homotopy equivalent to $\ground$.  
\end{rk}	
\section{Categories of twisted modules associated with simplicial sets}\label{simplicial}
In this section we consider twisted modules over the dg pseudo-compact algebra $C^*(X)$, the normalized cochain complex of a simplicial set $X$. We have the dg categories $\MCdg(C^*(X)), \Tw(C^*(X))$ and $\Twfg(C^*(X))$ that we will abbreviate to $\MCdg(X), \Tw(X)$ and $\Twfg(X)$ respectively. These dg categories are not (up to quasi-equivalence) invariants of the weak homotopy type of $X$, however they are homotopy invariants of $X$ in a sense that we will now make precise.

Recall cf. \cite{Goerss99} that two maps of simplicial sets $f,g:X\to Y$ are called homotopic if they can be extended to a map $X\times K_0\to Y$; recall that $K_0$ stands for the standard simplicial interval having two nondegenerate 0-simplices and one nondegenerate 1-simplex. This notion of homotopy is completely adequate only in the case where $Y$ is a Kan complex (in which case it is an equivalence relation). We will now introduce the notion of a \emph{strong homotopy} of maps between simplicial sets and the concomitant notion of strong homotopy equivalence. Let $C$ be a fibrant cylinder object for the simplicial point. For example, we can take $C=I$ or $C=K_\infty$. Then $X\times C$ is a cylinder object for any simplicial set $X$; moreover it is \emph{very good} in the sense that the natural projection $X\times C\to X$ is a fibration of simplicial sets. We will denote by $i_0, i_1:X\to X\times C$ the two natural inclusions of $X$ into $X\times C$.
\begin{defn}
	An elementary strong homotopy of maps of simplicial sets $f,g:X\to Y$ is a map $h:X\times C\to Y$ such that $h\circ i_0 =f$ and $h \circ i_1=g$.
	the maps $f$ and $g$ are called strongly homotopic if there is a chain of elementary homotopies connecting $f$ and $g$.
	
	Furthermore, $X$ and $Y$ are called strongly homotopy equivalent if there are maps $f:X\to Y$ and $g:Y\to X$ such that $f\circ g$ and $g\circ f$ are strongly homotopic to $\id_Y$ and $\id_X$ respectively.
\end{defn}
\begin{prop}
	The relation of strong homotopy does not depend on the choice of a very good cylinder object $C$. Any such very good cylinder is strongly homotopy equivalent to the point.
\end{prop}
\begin{proof}
Let $P$ stand for the simplicial point. For any two very good cylinder objects $C$ and $C^{\prime}$ of $P$ consider the diagram
\[\xymatrix{
P\coprod P\ar^{i_0\coprod i_1}[d]\ar^{i_0\coprod i_1}[r]&C^{\prime}\ar[d]\\
C\ar[r]&P.
}\]
Since the left downward arrow is a monomorphism and thus a cofibration of simplicial sets, and the right downward arrow is a fibration ($C^{\prime}$ being fibrant), there exists a filler $C\to C^{\prime}$. It follows that any strong homotopy based on $C^{\prime}$ gives rise to a strong homotopy based on $C$. Symmetrically, any strong homotopy based on $C$ gives rise to a strong homotopy based on $C^{\prime}$; this proves the first claim of the proposition.  The second claim follows from general theory of closed model categories: any very good cylinder object is weakly equivalent to the point; then, being a fibrant-cofibrant object it is homotopy equivalent to the point through any fixed good cylinder object, i.e. it is strongly homotopy equivalent to the point.
\end{proof}
\begin{rk}
	Two natural candidates for $C$ are $K_\infty$ and $I$, the singular simplicial set of the unit interval $[0,1]$. The multiplication on $[0,1]$ turns $I$ into a simplicial monoid, and the multiplication map $I\times I\to I$ could be viewed as a homotopy between the identity map on $I$ and the map to the point. This is an explicit strong homotopy equivalence between $I$ and the point.
\end{rk}
\begin{prop}\label{prop:strongsimplicial}
	Let $f,g:X\to Y$ be two maps of simplicial sets. If $f$ and $g$ are strongly homotopic, then the induced maps of pseudo-compact dg algebras
	$f^*,g^*:C^*(Y)\to C^*(X)$ are strongly homotopic.
	
	If two simplicial sets $X$ and $Y$ are strongly homotopy equivalent, then the pseudo-compact dg algebras $C^*(X)$ and $C^*(Y)$ are strongly homotopy equivalent.
\end{prop}
\begin{proof}
	The second statement of the proposition follows from the first. For the first, choosing $X\times I$ as a very good cylinder object for $X$, consider a homotopy $h:X\times I\to Y$ such that $h\circ i_0=f$ and $h\circ i_0=g$. This gives rise to a map of dg pseudo-compact algebras
	$C^*(Y)\to C^*(X\times I)$ and, composing the latter with the Eilenberg-Zilber map $C^*(X\times I)\to C^*(X)\otimes I^*$ (which is known to be a dg algebra map) we obtain the desired strong homotopy between $f^*$ and $g^*$.
\end{proof}
\begin{cor}\label{cor:equivalence}
Let $X$ and $Y$ be two strongly homotopy equivalent simplicial sets. Then the following dg categories are quasi-equivalent:
\begin{enumerate}
	\item $\MCdg(X)$ and $\MCdg(Y)$,
	\item $\Tw(X)$ and $\Tw(Y)$,
	\item $\Twfg(X)$ and $\Twfg(Y)$.
\end{enumerate}
\end{cor}
\begin{proof}
	This is a direct consequence of Theorem \ref{thm:homotopyequivalence} and Proposition \ref{prop:strongsimplicial}.
\end{proof}
\begin{cor}\label{cor:weaklyKan}
	Let $X$ and $Y$ be two weakly equivalent Kan simplicial sets. Then the following dg categories are quasi-equivalent:
	\begin{enumerate}
		\item $\MCdg(X)$ and $\MCdg(Y)$,
		\item $\Tw(X)$ and $\Tw(Y)$,
		\item $\Twfg(X)$ and $\Twfg(Y)$.
	\end{enumerate}
\end{cor}
\begin{proof}
	Two weakly equivalent Kan simplicial sets are homotopy equivalent through any given very good cylinder object. In particular, they are strongly homotopy equivalent. The conclusion then follows from Corollary \ref{cor:equivalence}.
\end{proof}
This also has a consequence for singular cochain algebras of topological spaces.
\begin{cor}\label{cor:mctopologicalequivalence}
Let $X$ and $Y$ be weakly equivalent topological spaces. Then there are quasi-equivalences of dg categories between
\begin{enumerate}
	\item $\MCdg(X)$ and $\MCdg(Y)$,
	\item $\Tw(X)$ and $\Tw(Y)$,
	\item $\Twfg(X)$ and $\Twfg(Y)$.
\end{enumerate}
In particular, if $X$ is a contractible	topological space, then the dg categories $\MCdg(X)$, $\Twfg(X)$ and $\Tw(X)$ are quasi-equivalent to the category of free $\ground$-modules of rank 1, the category of finitely generated free dg $\ground$-modules and the category of all free dg $\ground$-modules, respectively.

\end{cor}
\begin{proof}
Since the topological spaces $X$ and $Y$ are weakly equivalent, so are their singular simplicial sets. Since the latter are Kan complexes, the claim follows from Corollary \ref{cor:weaklyKan}.
\end{proof}
\subsection{Reduced and minimal twisted modules}
Let $A$ be a non-negatively graded pseudo-compact algebra, such as $C^*(X)$ for a simplicial set $X$, and $M:=V\otimes A$ be a twisted $A$-module. The differential $D_M$ on $M$ is determined by its restriction on $V$; furthermore we have: $D_M|_{V\otimes 1}=d^0+d^1+\ldots$ where $d^n:V\to V\otimes A^n$. In particular, $d^0:V\to V\otimes A^0$ determines an $A^0$-linear differential on $V\otimes A^0$.
\begin{defn}
	A twisted $A$-module $M$ as above is called \emph{reduced} if $d_0$ factors through $V\hookrightarrow V\otimes A^{0}: v\mapsto v\otimes 1$, i.e. if it is induced by a differential in the graded $\ground$-module $V$. If, further, $d_0=0$, we will call the twisted $A$-module $M$ \emph{minimal}. We will denote by $\Twred(A)$, $\Twm(A)$, $\Twfgm(A)$,  and $\Twfgred(A)$ the categories of reduced, minimal twisted $A$-modules and their perfect versions respectively. If $A=C^*(X)$ for a simplicial set $X$, we will denote these categories by $\Twred(X)$, $\Twm(X)$, $\Twfgm(X)$  and $\Twfgred(X)$ respectively.
\end{defn}

\begin{rk}
		If $A$, in addition to being non-negatively graded, is \emph{connected} i.e. $A^0=\ground$, then clearly any twisted $A$-module is reduced.
		Such is the case, when $A=C^*(X)$ for a reduced simplicial set $X$. 
\end{rk}
\begin{rk}	
		The notion of a minimal twisted module is similar to that of a minimal $A_\infty$-module, \cite{Keller01}; indeed in the case when $A$ is a completed tensor algebra representing an $A_\infty$ algebra, then a minimal twisted $A$-module is a contramodule corresponding to a minimal $A_\infty$-module under the comodule-contramodule correspondence, cf. \cite[Theorem 5.2]{Positselski11}.
\end{rk}
\begin{prop}
	A homotopy equivalence between two minimal twisted modules is necessarily an isomorphism.
\end{prop}
\begin{proof}
	It suffices to show that any endomorphism of a minimal twisted module that is homotopic to the identity is invertible. Let $V\otimes A$ be such a minimal $A$-module; then its dg algebra of endomorphisms is $A\otimes\End(V)$; by minimality the differential in it has the form
	$D_A=d_A^1+d_A^2+\ldots$
	where \[d_A^n|_{V\otimes 1}: V\to V\otimes A^n.\]
	Let $f\in A\otimes\End(V)$ be a closed endomorphism homotopic to the identity; thus $f=1+D_A(g)$ for some $g\in A\otimes\End(V)$. Then $D_A(g)$ must have the form $D_A(g)=d_A^1(g)+d_A^2(g)+\ldots$ with $d_A^n(g)\in A^n\otimes\End(V)$ and therefore $f$ is invertible: $f^{-1}=1+\sum_{i=1}^{\infty}(-1)^i(\sum_{n=1}^\infty d_A^i(g))$.
\end{proof}
The following result is analogous to the well-known theorem on minimal $A_\infty$ modules \cite{Keller01}.
\begin{prop}\label{prop:minimal}
	If $\ground$ is a field then any reduced twisted $A$-module is homotopy equivalent to a minimal one.
\end{prop}
\begin{proof}
	Let $A\otimes V$ be a reduced twisted $A$-module; it has differential $d^0+d^\prime:=d^0+d^1+\ldots$ as described above and $d^0$ makes $V$ into a dg $\ground$-vector space. Since $\ground$ is a field, $V$ admits a decomposition $V\cong H(V)\oplus d_0(V)\oplus U$ with $d^0$ mapping $U$ isomorphically onto $V$. Denote by $t:V\to V$ the projection onto $H(V)$ and by $s:V\to V$ the operator that is inverse to $d^0$ on $d^0(V)$ (viewed as an operator $U\to d^0(V)$) and whose restriction on $H(V)$ and $U$ is zero. The pair of operators $(s,t)$ determines an abstract
	Hodge decomposition on $V$ (cf. for example \cite{Chuang19} concerning this notion) and we can apply the Perturbation Lemma as formulated in \cite[Corollary 3.17]{Chuang19}. Namely, the twisted module $A\otimes V$ is isomorphic to the direct sum of $M_1:=A\otimes(d^0(V)\oplus U)$ and $M_2:=A\otimes H(V)$ where $M_1$ is supplied with the differential $\id_A\otimes d^0$ and $M_2$ with the differential $td^\prime(1+sd^\prime)^{-1}t$. Since $M_1$ is clearly homotopy equivalent to zero and $M_2$ is minimal, the claim follows.
\end{proof}
Then we have the following result.
\begin{cor}\label{cor:minimal}
	If $A$ and $B$ are two non-negatively graded pseudo-compact dg algebras that are strongly homotopy equivalent. Then the following categories are quasi-equivalent:
	\begin{enumerate}
		\item $\Twm(A)$ and $\Twm(B)$,
		\item $\Twfgm(A)$ and $\Twfgm(B)$.
	\end{enumerate}
	Let $A$ be a connected pseudo-compact dg algebra over a field. Then the following dg categories are quasi-equivalent:
	\begin{enumerate}
		\item $\Tw(A)$ and $\Twm(A)$,
		\item $\Twfg(A)$ and $\Twfgm(A)$.
	\end{enumerate}
\end{cor}
\begin{proof}
	Let $f:A\to B$ and $g:B\to A$ be two maps that are inverse up to $K_2$-homotopy. These maps induce dg functors $\Tw(f):\Tw(A)\to\Tw(B)$ and $\Tw(g):\Tw(B)\to \Tw(A)$ inducing quasi-equivalence of the corresponding dg categories by (the pseudo-compact analogue of) Theorem \ref{thm:homotopyequivalence}. These functors restrict to the categories of minimal twisted modules and, using Proposition \ref{prop:minimal} we see, that these restrictions give mutually inverse quasi-equivalences.  The statement about perfect minimal twisted modules is proved similarly.
	
	Finally, if $A$ is connected, any twisted $A$-module is automatically reduced, and the proof is finished by appealing to Proposition \ref{prop:minimal}. 	
\end{proof}
\begin{thm}\label{thm:connected}
	Let $X$ be a connected Kan simplicial set. Then the pseudo-compact dg algebra $C^*(X)$ is strongly homotopy equivalent to a connected one.
\end{thm}
\begin{proof}
	Choosing a vertex of $X$ amounts to constructing a map $P\to X$ from the one-point simplicial set $P$ to $X$.	Let $X^\prime$ be the simplicial set defined by the pullback diagram
	\[
	\xymatrix{
		X^\prime\ar^f[r]\ar[d]&X\ar[d]\\
		P\ar[r]&\operatorname{cosk_0}(X).	
	}
	\]
	Here $\operatorname{cosk_0}(X')$ is the 0-coskeleton (the zeroth stage of the Moore-Postnikov tower of $X$). The simplicial set $X^\prime$ has a single vertex corresponding to the map $P\to X$ and so $C^*(X^\prime)$ is a connected pseudo-compact dg algebra. It is well-known (e.g. \cite[Proposition 8.2, Theorem 8.4]{May67}) that $\operatorname{cosk_0}(X)$ is a weakly contractible Kan simplicial set, and it follows that $X^\prime$ is likewise Kan.   Then $f:X^\prime\to X$ is a strong homotopy equivalence with a strong homotopy inverse $g:X\to X^\prime$ (in fact it is clear that $X^\prime$ is a deformation retract of $X$ so that $g\circ f=\id_{X^\prime}$). By Proposition \ref{prop:strongsimplicial} $C^*(X)$ and $C^*(X^\prime)$ are strongly homotopy equivalent.
\end{proof}
Combining Corollary \ref{cor:minimal} and Theorem \ref{thm:connected} we obtain the following result.
\begin{cor}\label{cor-twistedreduced} \
	\begin{enumerate}\item
		Let $X$ be a Kan simplicial set. Then there is a quasi-equivalence between the following dg categories:
		\begin{enumerate}
			\item $\Tw(X)$ and $\Twred(X)$,
			\item $\Twfg(X)$ and $\Twfgred(X)$.
		\end{enumerate}
		\item If $\ground$ is a field then, additionally, the following dg categories are quasi-equivalent:
		\begin{enumerate}
			\item $\Tw(X)$ and $\Twm(X)$,
			\item $\Twfg(X)$ and $\Twfgm(X)$.
		\end{enumerate}
		\item If $X,X^\prime$ are weakly equivalent Kan simplicial sets, then the following dg categories are quasi-equivalent:
		\begin{enumerate}
			\item $\Twm(X)$ and $\Twm(X^\prime)$,
			\item $\Twfg(X)$ and $\Twfgm(X^\prime)$.	\qed
		\end{enumerate}
	\end{enumerate}
\end{cor}

\begin{rk}\label{rem:cobar}
	If $X$ is a reduced simplicial set and $\ground$ is a field then $C^*(X)$ is a local pseudo-compact dg algebra (which is the dual to a conilpotent dg coalgebra $C_*(X)$). The category of local pseudo-compact dg algebras admits the structure of a closed model category, see \cite{Positselski11} where this result is formulated in the dual language of coalgebras. It makes sense to ask whether weakly (or strongly) homotopy equivalent reduced simplicial sets give rise to weakly equivalent (in the sense of the aforementioned closed model category) local pseudo-compact dg algebras. A partial answer to this question could be extracted from the recent paper \cite{Zeina18} where it is proved that if $X$ is a singular simplicial set of a topological space (or, more generally, a Kan simplicial set) then the cobar-construction of $C_*(X)$ is quasi-isomorphic to the dg algebra of chains on the loop space of $X$. Note that this generalizes the classical result of Adams on the cobar-construction \cite{Adams56} in that simple connectivity of $X$ is not assumed. This result implies that for two weakly equivalent \emph{Kan} simplicial sets $X$ and $X^\prime$ the pseudo-compact local dg algebras $C^*(X)$ and $C^*(X^\prime)$ are indeed weakly equivalent. The Kan condition is essential; e.g. taking for $X$ a simplicial circle with one non-degenerate simplex in degree zero and another in degree one (which is not a Kan simplicial set), a straightforward inspection shows that the cobar-construction of $C_*(X)$ is isomorphic to $\ground[x]$ with $|x|=0$ whereas the singular chain algebra on $\Omega(S^1)=\mathbb Z$ is $\ground[x,x^{-1}]\neq \ground[x]$.
\end{rk}

\section{Twisted modules and sheaves}\label{sheaves}

\subsection{Generalities on dg sheaves}
In this section we will present the local to global arguments needed to apply Schlessinger-Stasheff type results to infinity local systems. The results obtained here are directly used in Section \ref{sect-derham} and Section \ref{sect-dolbeault} and the methods of proof are used in Section \ref{sect-singular}.

Let $X$ be a topological space, always assumed paracompact and Hausdorff in this section.
Let $\cat R$ be a sheaf of $k$-algebras on $X$ and let $\A = (\A^{\bullet}, d)$ be a sheaf of non-negatively graded dg $\cat R$-algebras. Write $\A^{\#}$ for $(\A^{\bullet}, 0)$.
We will consider the dg algebra $A \coloneqq \A(X)$.

Write $A\mods$ for the dg category of dg $A$-modules and $\A\mods$ for the dg category of sheaves of dg modules over $\A$. Write  $\cat R\mods$ for the dg category of sheaves of dg $\cat R$-modules.

There is an adjunction $p^{*}\dashv p_{*}$ between dg modules over $\A$ and dg modules over $A$ which is induced by the map $p: (X,\A) \to (*, A)$ of dg ringed spaces.

For free modules we recall the following straightforward result:
\begin{lemma}\label{lemma-protoswan}
	The adjunction $p^{*} \dashv p_{*}$ induces a dg equivalence between dg $A$-modules and dg $\A$-modules whose underlying $A^{\#}$-modules, respectively $\cat A^{\#}$-modules, are finitely generated and free. 
\end{lemma}
\begin{proof}
	We first forget the differential and the grading and consider a ringed space $(X, \cat A)$ and let $A = \cat A(X)$. Let $q$ be the map $(X, \cat A)\to (*, A)$ Then $q^{*}$ induces an equivalence between finitely generated free $A$-modules and finitely generated free $\cat A$-modules:
	Any free $\cat A$-module is of the form $q^{*}V$ 
	for a free $A$-module $V$ and $q_{*}q^{*}(W) = (q^{*}W)(X) = W$ as $q_*$ commutes with finite direct sums,
	 which gives the isomorphism $\Hom_{\cat A}(q^{*}V, q^{*}W) \cong \Hom_{A}(V, W)$.
		
		Thus the unit and counit of the adjunction are isomorphisms. They are also compatible with the grading and the differential, hence they are isomorphisms of dg modules, resp.\ dg sheaves, proving the lemma.
	\end{proof}

For a fine sheaf $\cat R$ we can compare categories of locally free sheaves and projective modules over the ring of global sections.

Recall that a sheaf $\cat F$ is \emph{fine} if for any locally finite open cover $\{U_{i}\}$ of $X$ there is family of morphisms $\phi_{i}: \cat F \to \cat F$ such that $\sum \phi_{i} = \id_{\cat F}$ and such that $\phi_{i}$ has support contained in $U_{i}$.
On a paracompact Hausdorff space fine sheaves are always soft and thus $\Ga$-acyclic. A module over a fine sheaf of rings is automatically fine. For more details see e.g. \cite[Section II.3]{Wells07}.

The following is Swan's theorem as stated in \cite[Corollary 3.2]{Morye09}.

\begin{thm}\label{thm-swan}
Assume $(X, \cat R)$ is a locally ringed space of finite covering dimension with $\cat R$ a fine sheaf of commutative algebras. Then the category of finitely generated projective $\cat R(X)$-modules is equivalent to the category of locally free $\cat R$-modules of bounded rank.
\end{thm}

We note that Lemma \ref{lemma-protoswan} gives an equivalence of dg categories, but the two sides have a priori very different homotopy theories: For dg sheaves the natural class of weak equivalences is given by local quasi-isomorphisms, i.e.\ morphisms which restrict to quasi-isomorphisms on all stalks. 

To make this more precise we recall that the categories $\cat R\mods^{psh}$ and $\cat A\mods^{psh}$ of presheaves of dg $\cat R$-modules, respectively presheaves of dg $\A$-modules, have model structures.

To define this, first recall the definition of a \emph{hypercover}. A hypercover of a presheaf $P$ on a topological space $X$ is an augmented simplicial presheaf $C_{\bullet} \to P$ such that all $C_{n}$ are coproducts of representable presheaves and for all $n$ the map $C_{n} \to (\operatorname{cosk}_{n-1}C)_n$ is a cover, where we take the coskeleton in the augmented sense.
In particular a hypercover of $X$ is defined to be a hypercover of the presheaf $h_{X}$ that $X$ represents, and it may be represented by a cover $\mathfrak U_{0} \to X$ together with covers $\mathfrak U_{n} \to \lim_{k < n} \mathfrak U_{k}$. The fundamental example of a hypercover is the nerve of a \v Cech cover. In this case all the covers (except for $\mathfrak U_{0} \to X$) are isomorphisms.

\begin{defn} The \emph{projective model structure} on presheaves of dg $\cat R$-modules has fibrations and weak equivalences defined object-wise. The \emph{local model structure} on presheaves of dg $\cat R$-modules is the localization of the projective model structure at all hypercovers.
\end{defn}

 Then weak equivalences are given by maps inducing weak equivalences on stalks. We will abuse notations and refer to them as quasi-isomorphisms. An object $P$ is fibrant if it is a \emph{hypersheaf}, i.e.\ for any open subset $U \subset X$ and hypercover $\mathfrak U_{\bullet} \to U$ there is a quasi-isomorphism $P(U) \simeq \check C(\mathfrak U_{\bullet}, P)$.
Here the right hand side is the \v Cech complex of a hypercover, defined exactly like the \v Cech complex for a cover. We say a hypercover is \emph{contractible} if every $\mathfrak U_{n}$ is a coproduct of contractible open sets. Any locally contractible topological space admits a contractible hypercover.

\begin{rk}
We will in the following sometimes compute \v Cech cohomology with respect to a hypercover, but not much generality is lost if the reader wants to mentally replace them by \v Cech covers.
\end{rk}

The local model structure on presheaves of dg $\cat A$-modules is defined in the same way (or it can be transferred via the forgetful functor, see e.g.\ \cite{Goerss06}).

The homotopy category of this model category is the usual derived category of $\cat R\mods$, respectively $\A\mods$. The adjunction $p_{*} \dashv p^{*}$ is Quillen.

For more details on the local model structure see \cite{Dugger04a} (in the case of simplicial presheaves) and \cite{Choudhury19} (for chain complexes).

We now assume $\cat R \to \cat A$ is an object-wise quasi-isomorphism, i.e.\ in particular a local weak equivalence.
Writing $J$ for the forgetful functor we have a dg adjunction $-\otimes_{\cat R}\A \dashv J: \cat R\mods \rightleftarrows \A\mods$. We consider the derived categories of $\cat R\mods$ and $\cat A\mods$ and write $RJ$ for the total derived functor of $J$, i.e.\ the lift of $J$ to the derived categories. As $\A$ and $\cat R$ are quasi-isomorphic $RJ$ is an equivalence.

\begin{defn}\label{defn-F}
Let $F = J \circ p^{*}: A\mods \to \cat R\mods$ be the dg functor given by composition of the two functors defined above. 
\end{defn}
We will use $F$ to map different categories of twisted modules to dg sheaves.
In the remainder of this paper we will abuse notation and write $F$ for different choices of $A$ as well as for the restriction of $F$ to $\Tw(A)$, $\Twfg(A)$ and $\Twfgfree(A)$. 

\subsection{Twisted modules and perfect complexes}\label{sect-cohesivesheaves}

We now consider the functor $F: A\mods \to \cat R\mods$ in more detail. We let $D(X, \cat R)$ or simply $D(\cat R)$ be the derived category of $\cat R\mods$.

In this section we will assume $X$ is \emph{locally good}, which is defined as follows. 
We say a ringed space $(U, \cat R)$ is \emph{good} if the natural map $\mathcal{R}(U)\to Ru_{*}\cat R|_{U}$ is a quasi-isomorphism (here $u:U\to *$ is the map to the one-point space). Then $X$ is locally good if its topology has a basis of good open sets.

Most spaces of interest are locally good, for example algebraic schemes, analytic spaces and locally contractible topological spaces with the constant structure sheaf. Good neighbourhoods are given by affine subvarieties, Stein subspaces and contractible subsets respectively.

\begin{defn}
A dg $\cat R$-module on $X$ is \emph{strictly perfect} if it is bounded and a direct summand of a free sheaf of finite rank in each degree.
A dg $\cat R$-module $G$ is \emph{perfect} if for every $x \in X$ there is a neighbourhood $U$ such that $G|_{U}$ is quasi-isomorphic to a strictly perfect dg sheaf.
\end{defn}

\begin{rk}\label{rk-locallyfree}
	If $(X, \cat R)$ is locally ringed than a dg $\cat R$-module is
	perfect if it is locally quasi-isomorphic to a bounded complex of free sheaves of finite rank in each degree.
	This follows as any direct summand of a free $\cat R$-module is locally free, see \cite[Tag 01C5]{Stacks13}.	
\end{rk}

We denote  by $\Dperf(X, \cat R)$ or $\Dperf(\cat R)$ the subcategory of $D(X, \cat R)$ consisting of perfect dg sheaves of $\cat R$-modules.
We will say a perfect dg sheaf of $\cat R$-modules is \emph{globally bounded} if there are integers $a, b$ and $N$ such that there is a cover $\{U_{i}\}$ such that each $G|_{U_{i}}$ is quasi-isomorphic to a strictly perfect dg sheaf $G^{U}$ which is concentrated in degrees $[a, b]$ and has at most $N$ generators. We let $\Dperf^{B}(\cat R)$ denote the subcategory of globally bounded perfect dg sheaves.

\begin{rk}
In many cases of interest all perfect dg sheaves are globally bounded. An example of a non-globally bounded one is given by the following construction. Consider $\C$ equipped with the holomorphic (or smooth) structure sheaf. Then the skyscraper sheaf $\C_{n}$ at $n \in \C$ is perfect and so is $\oplus_{n \in \set N} \C_{n}^{\oplus n}$. But this sum is clearly not globally bounded.
\end{rk}

For later use we also define $\Dlf(X, \cat R)$ or $\Dlf(\cat R)$ to be the subcategory of $D(\cat R)$ consisting of locally free dg sheaves of $\cat R$-modules, i.e.\ those which are locally quasi-isomorphic to free $\cat R$-modules without any finiteness assumptions. In the case $\cat R=\underline{\ground}$, the locally free dg sheaves of $\cat R$-modules will be referred to as cohomologically locally constant (clc) sheaves.
We will need the following:
\begin{lemma}\label{lemma-perfectkaroubi}
$\Dlf(\cat R)$, $\Dperf(\cat R)$ and $\Dperf^{B}(\cat R)$ are idempotent complete.
\end{lemma}
\begin{proof}
The result for $\Dlf(\cat R)$ follows from \cite{Bokstedt93}.

Next, recall that for any ring perfect dg modules are exactly compact objects in the derived category, and since compact objects are closed under direct summands so are perfect dg sheaves, see e.g.\cite[Proposition 6.4]{Bokstedt93}.

We consider a perfect dg sheaf of the form $G \simeq M \oplus N$ and will now show $M$ is perfect.
It follows from the definition that any point has a neighbourhood $U$ on which we may assume $G$ is strictly perfect. Then the restriction $G|_{U}$ is isomorphic to the sheaf associated to $G(U)$ (apply Lemma \ref{lemma-protoswan} and extended to the idempotent completion). 

We write $u: U \to *$.  Then $G|_{U} \cong u^{*}G(U)\simeq Lu^{*}G(U)$ as $G(U)$ is cofibrant.
We may assume $U$ is good and then $G(U) \simeq Ru_{*}G \simeq Ru_{*}M\oplus Ru_{*}N$. $G(U)$ is perfect and thus so is $Ru_{*}M$. As $G \simeq Lu^{*}G(U)$ it follows that $M \oplus N \simeq Lu^{*}Ru_{*}M \oplus Lu^{*}Ru_{*}N$. Since the map $M \to Lu^{*}Ru_{*}N$ in the derived category corresponds to the zero map $Ru_{*}M \to Ru_{*}N$ under an adjunction it is zero and $M \simeq Lu^{*}Ru_{*}M$.
Thus $M$ is perfect. The same argument applies to globally bounded dg sheaves.
\end{proof}
\begin{prop}\label{prop-cohesivesheaves}
Let $(X, \cat R)$ be a ringed space, and let $\A$ be a fine sheaf of dg algebras on $X$ such that there is a quasi-isomorphism $\cat R \to \A$ satisfying condition (*) below.
Then the associated sheaf functor $p^{*}$ gives a quasi-fully faithful functor $\Twfg(A) \to \A\mods$ and $F = J\circ p^{*}$ induces an embedding of triangulated categories $\Hot(\Twfg(A)) \to \Dperf^{B}(X, \cat R)$.
\end{prop}

\begin{rk}
	Proposition \ref{prop-cohesivesheaves} only depends on the construction of $\Twfg(A)$ up to quasi-equivalence and thus holds equally if we consider $A$ as a dg AM algebra, see Remark \ref{rk:twfgam}.
\end{rk}
The crucial assumption is the following 
\begin{itemize}
\item[(*)] For every free graded $k$-module $G$, every $x \in X$ with a neighbourhood $U'$ and any MC element $\xi \in \MC(\A(U')\otimes \uEnd(G))$ there is a neighbourhood $x\in U \subset U'$ such that $\xi|_{U}$ is homotopy gauge equivalent to an element in the image of the Maurer Cartan set of $\cat R(U)\otimes \uEnd(G)$.
\end{itemize}

We will be particularly interested in cases where (*) is the consequence of the following stronger condition:
\begin{itemize}
\item[(**)] For every $x \in X$ and every free graded $k$-module $G$ there is a neighbourhood $U$ such that $\cat R \to \A$ induces a quasi-equivalence $\MCdg(\cat R(U)\otimes \uEnd(G)) \simeq \MCdg(\A(U)\otimes \uEnd(G))$.
\end{itemize}

\begin{proof}[Proof of Proposition \ref{prop-cohesivesheaves}]

By Lemma \ref{lemma-protoswan} the restriction $p^{*}: \cat \Twfgfree(A) \to \A\mods$ is quasi-fully faithful.
Then $p^{*}$ on $\Twfg(A)$ is just the extension of $p^{*}|_{\Twfgfree(A)}$ to homotopy idempotents and it follows
that $p^{*}$ is also quasi-fully faithful.

It remains to prove the statement on homotopy categories. We have the following composition
\[ \Hot(\Twfg(A)) \stackrel {p^{*}}\longrightarrow \Hot(\A\mods) \stackrel {q_{\A}} \longrightarrow D(\A) \stackrel {RJ} \longrightarrow D(\cat R)\]
where $q_{\A}$ is the quotient by quasi-isomorphisms.

In Lemma \ref{lemma-fullyfaithful} we will show that  $q_{\A}$ is fully faithful on the image of $p^{*}$. It is well-known that $RJ$ is fully faithful. Thus $RJ \circ q_{\A} \circ \Hot(p^{*})$ is fully faithful. It is clearly compatible with shifts and cones.

The fact that $RJ$ lands in $\Dperf(\cat R)$ is Lemma \ref{lemma-mc}.
\end{proof}

\begin{lemma}\label{lemma-mc}
The dg functor $F$ sends perfect twisted $A$-modules to globally bounded perfect sheaves of $\cat R$-modules.
\end{lemma}
\begin{proof}
By Lemma \ref{lemma-perfectkaroubi} it suffices to show that a finitely generated twisted $A$-module $E$ is sent to a globally bounded perfect dg sheaf.

We may write $E = (G\otimes A, D)$ where $G$ is a free graded module over $k$. It suffices to show  that $F(E) \cong (\underline G \otimes \cat A, D)$ is perfect locally.

On any $U$ we know that  $D|_{U} - \id\otimes d_{A(U)}$ can be represented by a MC element $\xi$ in $ \A(U) \otimes \uEnd(G)$. Fix some $x$. For a suitably small neighbourhood we may assume that $\xi$ is as in in condition $(*)$.
Thus there is a homotopy gauge equivalence $g \in \cat A(U)\otimes \uEnd(G)$ between $\xi$ and some element $\eta$ in the image of $\cat R(U)\otimes\uEnd(G)$.

It follows that $g$ gives a homotopy equivalence from $(G \otimes \cat A(U), D_{U})$ to $(G \otimes \cat A(U), d_{G}\otimes \id + \id \otimes d_{\cat A})$ where $d_{G}$ is some differential on $G \otimes \cat R(U)$. Thus  we obtain a quasi-isomorphism from $F(E)|_{U}$ to the perfect dg sheaf $(\underline G\otimes \cat R|_{U}, d_{G})$ of $\cat R|_{U}$-modules.

Boundedness follows immediately from finite generation of $E$.
\end{proof}

\begin{lemma}\label{lemma-fullyfaithful}
The natural functor $q_{\cat A}: \Hot(\A\mods)\to D(\A)$ is fully faithful when restricted to the image of perfect twisted modules.
\end{lemma}
\begin{proof}
It suffices to consider finitely generated twisted modules, so we fix $(V \otimes A, D_{V})$, $ (W \otimes A, D_{W}) \in \cat \Twfgfree(A)$ and compute $R\uHom_{D(\A)}(\underline V\otimes \A, \underline W\otimes \A)$. The derived Hom can be computed as derived global sections of the sheaf Hom $U \mapsto R\uHom_{\A|_{U}}((\underline V\otimes \A)|_{U}, (\underline W \otimes \A)|_{U})$.

We first compute locally. We write $\cat V = \underline V\otimes \cat R$ and $\cat W = \underline W \otimes \cat R$.
Then let $U$ be any good open set as in condition (*), such that $(\underline V\otimes \A)|_{U}$ is homotopy equivalent to $(\cat V|_{U}, d_V)\otimes_{\cat R|_{U}} \A|_{U}$, say.
Then we can compute:
\begin{align*}
R\uHom_{\A|_{U}}((\underline V\otimes \A)|_{U}, (\underline W \otimes \A)|_{U})
& \simeq R\uHom_{\cat R|_{U}}((\cat V|_{U}, d_V), (\underline W\otimes \A)|_{U}) \\
& \simeq \uHom_{\cat R|_{U}}((\cat V|_{U}, d_V), (\underline W\otimes \A)|_{U}).
\end{align*}
As $(\cat V|_{U}, d_{V})$ is free it is a cofibrant dg sheaf over $\cat R$ and the Hom space is underived.

To compute global sections we pick a hypercover $\mathfrak U$ consisting of good open sets $U$ satisfying condition (*).
By the above the Hom presheaf on $\mathfrak U$ may be written as $U \mapsto \uHom_{\cat R|_{U}}(\cat V, \cat W)\otimes_{\cat R(U)} \cat \A(U)$ with a suitable differential.

We compute \v Cech cohomology. 
Since $\A^{\#}$ is fine each $(\sHom_{\cat R}(\cat V, \cat W) \otimes_{\cat R} \A)^{i}$ has no higher cohomology. 
We filter the map $\epsilon: \check C^{*}(\mathfrak U, \sHom_{\cat R}(\cat V, \cat W) \otimes_{\cat R} \A) \to \check C^{0}(\mathfrak U, \sHom_{\cat R}(\cat V,\cat W)\otimes_{\cat R} \A)$ by the degree of coefficients. The associated map $\operatorname{Gr}(\epsilon)$ is a quasi-isomorphism and, since the filtration is exhaustive and Hausdorff, $\epsilon$ is a quasi-isomorphism, too.
Putting all of this together we have:
\begin{align*}
R\uHom_{D(\cat A)}(\underline V \otimes \cat A, \underline W\otimes \cat A)
&\simeq \check C^{*}(\mathfrak U, R\sHom_{\A}(\underline V\otimes \A, \underline W\otimes \cat A))  \\
&\simeq \check C^{*}(\mathfrak U, \sHom_{\cat R}(\cat V, \underline W\otimes \cat A))  \\
&\simeq \check C^{*}(\mathfrak U, \sHom_{\cat R}(\cat V, \cat W) \otimes_{\cat R} \A)  \\
&\simeq \check C^{0}(\mathfrak U, \sHom_{\cat R}(\cat V, \cat W) \otimes_{\cat R} \A)  \\
&\simeq \uHom_{\cat A\mods}(\underline V\otimes \cat A, \underline W\otimes \cat A). \qedhere
\end{align*}
\end{proof}

\begin{cor}\label{cor-quasifullyfaithful}
	If $\cat A$ is projective over $\cat R$ then $F: \Twfgfree(A) \to \cat R\mods$ is quasi-fully faithful.
\end{cor}
\begin{proof}
	We use the notation from the proof of Lemma \ref{lemma-fullyfaithful}.
	It suffices to compare $\uHom_{\cat R}(\underline V\otimes \cat A, \underline W\otimes \cat A)$ and $\uHom_{\cat A}(\underline V\otimes \cat A, \underline W\otimes \cat A)$. 
	Locally on $U$  the terms are quasi-isomorphic to $\uHom_{\cat R}((\cat V, d_{V})\otimes \cat A, \underline W\otimes \cat A)$ and $\uHom_{\cat R}((\cat V, d_{V}), \underline W\otimes \cat A)$, respectively. If $\cat A$ is projective over $\cat R$ then we ay replace $\uHom$ and $\otimes$ by their derived version and the two terms are quasi-isomorphic. The local-to-global argument remains unchanged.
\end{proof}
\begin{lemma}\label{lemma-surjective}
Let $(X, \cat R)$ and $\cat A$ be as in Proposition \ref{prop-cohesivesheaves}. If moreover $(X, \cat A^{0})$ is locally ringed, $\cat A^{0}$ is commutative, $A$ is flat over $A^{0}$ and $X$ has finite covering dimension then $\Dperf(X, \cat R)$ lies in the image of $\Hot(F)$.
\end{lemma}

\begin{proof}
	Consider a globally bounded perfect dg sheaf $\cat V$ of $\cat R$-modules on $X$. 
	Let $\Ga_{\cat V} \coloneqq (\Ga(X, {\cat V}\otimes_{\cat R} \A), D_{{\cat V}}\otimes \id + \id \otimes d_{A})$.
	This is a dg sheaf of $A$-modules which is not necessarily a perfect twisted module. However, it is a dg $A$-module of the form $Q\otimes_{A^{0}} A$ where $Q$ is some dg $A^{0}$-module. Such objects are called \emph{quasi-cohesive modules} in \cite{Block10}.
	
	By \cite[Lemma 2.3 and Proposition 2.5]{Morye09} the associated sheaf functor will send $\Ga_{\cat V}$ to ${\cat V}$.
	We will show that $\Ga_{\cat V}$ is homotopy equivalent to perfect twisted module  $\Ga_{{\cat V}}'$. 
	Then it is clear that $F(\Ga_{{\cat V}}')$ is homotopy equivalent to ${\cat V}$, proving essential surjectivity.
	
	By \cite[Theorem 3.2.7]{Block10} the quasi-cohesive module $\Ga_{\cat V}$ is homotopy equivalent to a cohesive module if  $\Ga(X, {\cat V}\otimes_{\cat R} \A^{0})$ is a perfect dg module over $A^{0}$. Here we use flatness of $A$ over $A^0$.
	
	It follows from our assumptions that ${\cat V} \otimes_{\cat R} \cat A^{0}$ is a globally bounded perfect dg sheaf of $\A^{0}$-modules, so by Lemma \ref{lemma-fineperfect} below
	it is quasi-isomorphic to a strictly perfect complex $P$.
	Thus by Theorem \ref{thm-swan} the global sections of $P$ are a perfect $\mathcal A^0(X)$-module quasi-isomorphic to $\Gamma(X, \mathcal V \otimes_{\cat R} \cat A^0)$.
	
	Now $\Gamma_{\cat V}$ is homotopy equivalent to a cohesive module und by Proposition \ref{prop:retract} it is also homotopy equivalent to a perfect twisted module.
\end{proof}
We needed the following  key lemma, which may be of independent interest.
The proof follows a standard method going back to \cite[Exposé II]{Berthelot06}
but over a fine structure sheaf one can avoid the usual assumption of compactness. 
The proof of (1) was shown to us by Zhaoting Wei.
\begin{lemma}\label{lemma-fineperfect}
	Let $(X, \mathcal A^0)$ be a locally ringed space with $\cat A^0$ commutative and fine and such that $X$ is of finite covering dimension. 
	Then
	\begin{enumerate}
		\item any globally bounded perfect dg sheaf of $\mathcal A^0$-modules $E$ is quasi-isomorphic to a bounded above complex $P$ of locally free sheaves of finite type,
		\item moreover, $P$ may be chosen to be a strictly perfect complex.
	\end{enumerate}
\end{lemma}
\begin{proof}
	We prove (1). Let $\mathcal A$ be the abelian category of $\mathcal A^0$-modules with the  subcategory $\mathcal D$ of finitely generated locally free $\mathcal A^0$-modules.
	In the derived category of $\mathcal A^0$-modules we consider the subcategory $\mathcal C$ of globally bounded perfect sheaves.
	
	Then we are in the setting of Lemma 1.9.5 of \cite{Thomason90} with $F: \mathcal D \to \mathcal A$ given by the inclusion.
	The lemma states that if $\mathcal D$ has enough objects to resolve cohomology sheaves, then any object in $\mathcal C$ is quasi-isomorphic to a bounded above chain complex $P$ over the category $\mathcal D$.
	
	The main part of this proof is devoted to showing the resolution property in our setting.
	We first give the precise statement needed to apply \cite[Lemma 1.9.5]{Thomason90}:
	Let $E$ be in $\mathcal C$ and $n$ be such that $H^i(E) = 0$ for $i > n$. 
	Let $A \to H^n(E)$ be an epimorphism of $\mathcal A^0$-modules. Then there is a $D \in \mathcal D$ and $D \to A$ such that the composition $D \to H^n(E)$ is an epimorphism.

	So let $E$ be a globally bounded perfect complex on $X$.
	By assumption there is a cover $\{U_{i}\}_{i \in I}$ of $X$ such that for each $i$ we have a quasi-isomorphism $\alpha_i: F_{i} \to E|_{U_{i}}$ from a strictly perfect dg sheaf on $U_i$.
	As $(X, \mathcal A^0)$ is locally ringed we may assume the $F_i$ are bounded complexes of finitely generated free $\cat A^0$-modules.
	
	Applying the canonical truncation functor to $F_i$ if necessary we may assume that the $F_i$ are concentrated in degree at most $n$.
	
	Next, following the argument in \cite[Proposition III.4.1]{Wells07} we may use a partition of unity to modify our cover to assume there is finite covering $\{U_i\}_{i \in I}$ (where each $U_i$ may have infinitely many connected components) such that each $F_i$ is a complex of free $\cat A^0$-modules.
	As $E$ is globally bounded we may also assume that each $F_i$ is finitely generated and bounded.
	
	Since each degree component of $F_i$ is free we can extend it to a finitely generated free sheaf on $X$.
	It is not in general possible to extend the differentials, but we use the following trick:
	for any $U_i$ we choose a proper open subset $W_i$, big enough such that $\{W_i\}$ is still an open covering. 
	Then using again that $\mathcal A^0$ is fine there is a section $\phi_i$ of $\mathcal A^0$ with support contained in $U_i$ that is equal to 1 on $W_i$.
	So we may define a new differential $\phi_i d$ on $F_i$ and extend it by 0 to obtain a global differential on the extension of $F_i$ to $X$. We denote this global complex of sheaves by $G_i$ and observe that $G_i|_{W_i}$ agrees with $F_i$.
	We need to define a map $\gamma_i: G_i \to E$. 
	Using that $F_i$ is bounded we may define $\gamma_i$ on $U_i$
	by $\phi_i^{k+1} \alpha_i$ in degree $m+k$, where $m$ is such that $F_i$ is 0 in degrees less than $m$.
	We then extend $\gamma_i$ by zero outside of $U_i$.
	
	By construction $G_i^k = 0$ if $k > n$, so $\gamma_i^n$ maps to $H^n(E)$ and as $\gamma_i|_{W_i}$ agrees with $\alpha_i$ it is a quasi-isomorphism for any $i$.
	Hence we obtain surjections $\gamma_i^n|_{W_i}$ which assemble to give a surjection 
	\[
	\gamma^n \coloneqq \sum_i \gamma_i^n: G \coloneqq \bigoplus_{i \in I} G_i^n \to H^n(E)
	\]
	where the sum is finite and thus $G$ is in $\mathcal D$.
	Since $G$ is free we can lift $\gamma^n$ to any epimorphism $A \to H^n(E)$.
	This shows the resolution property, so given $E$ we may construct a quasi-isomorphic bounded above complex $P$ of finitely generated locally free sheaves. This proves part (1).
	
	We prove (2).
	We have to show that $P$ may be chosen bounded below.
	We do this by considering its projective amplitude.
	
	We observe first that as $\A^0$ is fine, free $\mathcal A^0$-modules are projective objects in the abelian category of $\mathcal A^0$-modules.
	Conversely, since $(X, \mathcal A^0)$ is local any projective $\A^0$-module is locally free since it is a direct summand of a free $\A^0$-module, cf.\ Remark \ref{rk-locallyfree}.
	
	Next we note that there is $K$ such that for any $\mathcal A^0$-module $M$ the groups $\Ext^i(P, M)$ vanish for $i > K$.
	These Ext groups may be computed using a \v Cech resolution with open cover $\{W_i\}_{i\in I}$. 
	On each $W_i$ the complex $E|_{W_i}$ is represented by a perfect complex concentrated in degrees at least $n_i$, say.
	So for $K > |I|+ \max(-n_i)$ we obtain the desired vanishing.
	
	Now as $\cat A^0\mods$ has enough injectives and projectives the argument in \cite[Tag 0A5M]{Stacks13} applies and $P$ may be represented by a bounded below complex.
\end{proof}

We can now compare perfect twisted modules with perfect dg sheaves. The following two results are needed in the next section.

\begin{thm}\label{thm-twistedsheaves}
Let $(X, \cat R)$ and $\cat A$ be as in Proposition \ref{prop-cohesivesheaves}. If moreover $(X, \cat A^{0})$ is locally ringed, $\cat A^{0}$ is commutative, $A$ is flat over $A^{0}$ and $X$ has finite covering dimension then $F$ induces an equivalence $\Hot(\Twfg(A)) \to \Dperf^{B}(X, \cat R)$.
\end{thm}
\begin{proof}
This is Proposition \ref{prop-cohesivesheaves} together with Lemma \ref{lemma-surjective}, which says that the functor $\Hot(\Twfg(A)) \to \Dperf^{B}(X, \cat R)$ is essentially surjective.
\end{proof}

\begin{thm}\label{thm-twistedsheaves2}
Let $(X, \cat R)$ and $\cat A$ be as in Proposition \ref{prop-cohesivesheaves}. If moreover $(X, \cat A^{0})$ is locally ringed, $\cat A^{0}$ is commutative and $\cat R$ is the constant sheaf $\underline k$ then $F$ induces an equivalence $\Hot(\Twfg(A)) \to \Dperf(X, \underline k)$.
\end{thm}
\begin{proof}
Again this follows by Proposition \ref{prop-cohesivesheaves} together with essential surjectivity. By Lemma \ref{lemma-perfectconstant} below we may identify perfect dg sheaves with dg sheaves with locally constant cohomology.
Then we note that any locally constant sheaf $\cat M$ on $X$ is in the essential image of $F$. $\cat  M\otimes \A^{0}$ is locally free and thus $\Ga(X, \cat A \otimes \cat M)$ is a cohesive module by Theorem \ref{thm-swan} and thus homotopy equivalent to a perfect twisted module by Proposition \ref{prop:retract}.
Clearly $\cat M$ is quasi-isomorphic to $F(\Ga(X, \cat A \otimes \cat M))$.

As $F$ is quasi-fully faithful and $\Hot(\Twfg(A))$ is triangulated this shows that any subcategory of $\Dperf(X, \underline k)$ containing locally constant sheaves  and closed under triangles is in the essential image of $\Hot(F)$. But any perfect complex over $\underline k$ is a finite extension of its cohomology sheaves, thus contained in the image of $\Hot(F)$. We note that the global boundedness condition on perfect dg sheaves is automatic for clc sheaves on a connected space.
\end{proof}
\begin{lemma}\label{lemma-perfectconstant}
Let $X$ be locally contractible. Then $\Dlf(X, \underline{\ground})$ is equivalent to the derived category of clc sheaves. Moreover, $\Dperf(X, \underline k)$ is equivalent to the category of sheaves with locally constant cohomology sheaves whose fibres are perfect when considered as dg modules over $k$.
\end{lemma}
\begin{proof}
The cohomology of $\cat M \in \Dlf(X, \underline k)$ is locally given as the cohomology of a complex of $\ground$-modules, and  thus constant. 

Conversely consider a dg sheaf $\cat M$  and some contractible open set $U$ on which its cohomology is a constant $\underline k$-module. As $U$ has no cohomology $\cat M|_{U}$ is quasi-isomorphic to a direct sum of its cohomology sheaves.
Using free resolutions of the cohomology sheaves shows that $\cat M$ is locally quasi-isomorphic to a free dg sheaf.

The statement for perfect dg sheaves follows similarly.
\end{proof}
\begin{cor}\label{cor-cohesiveperfect}
In the setting of Theorem \ref{thm-twistedsheaves} and Theorem \ref{thm-twistedsheaves2} we also have $\Hot(\cat P_{A}) \cong \Dperf^{B}(X, \cat R)$.
\end{cor}
\begin{proof}
	This follows from the equivalence of cohesive modules and twisted modules provided by Corollary \ref{cor-twistedcohesive}.
\end{proof}
\begin{rk}
Corollary \ref{cor-quasifullyfaithful} shows that if we assume $\cat A$ is projective over $\cat R$ then moreover $F$ is quasi-fully faithful in Theorems \ref{thm-twistedsheaves} and \ref{thm-twistedsheaves2} and gives a quasi-equivalence with a dg category of perfect complexes.

In general, the equivalences of homotopy categories may be be enhanced to quasi-equivalences of dg categories between $\Twfg(A)$ and the dg-category of fibrant cofibrant dg $\cat R$-modules which are perfect dg sheaves. As presheaves in the image of $F$ are fibrant it suffices to compose $F$ with functorial cofibrant replacement.
\end{rk}

\section{Applications}\label{applications}
\subsection{The de Rham algebra}\label{sect-derham}
In this section the ground ring $k$ is $\set R$ and  $X$ is a connected smooth manifold. We consider perfect twisted modules over the de Rham algebra $\Om(X)$. We denote by $\Om$ the dg sheaf of de Rham algebras.

Recall that we consider $\Om(X)$ as a dg AM algebra and that all tensor products are understood to be completed.

Using what we have done so far we can recover and generalise the main result of \cite{Block09}, up to replacing infinity local systems by clc sheaves with cohomology sheaves of finite rank.

\begin{rk}
Note that one may consider cohesive modules (or equivalently perfect twisted modules) over the de Rham algebra $A$ as \emph{$\Z$-graded connections}. By Theorem \ref{thm-swan} we may consider a a complex of finitely generated projective $\Om^{0}(X)$-modules $E$ as a dg vector bundle $\cat E$, and the differential becomes a
 $\set Z$-graded connection $\set E: \cat E \to \cat E \otimes_{\Om^{0}} \Om$ satisfying $d\set E + \set E^{2} = 0$. This is the natural derived analogue of a vector bundle with a flat connection.
\end{rk}

\begin{thm}\label{thm-derham}
Let $X$ be a connected manifold (not necessarily compact).
Then the dg functor $F:\Twfg(X) \to \underline{\set R}\mods$ sending  $E$ to $U \mapsto E \otimes_{\Om(X)} \Om(U)$ 
is quasi-fully faithful and induces an equivalence $\Hot(\Twfg(X)) \cong \Dperf(X, \underline{\set R})$.
\end{thm}
\begin{proof}
The  theorem follows from Theorem \ref{thm-twistedsheaves} or Theorem \ref{thm-twistedsheaves2} together with Corollary \ref{cor-quasifullyfaithful}, applied to $\cat R = \underline {\set R}$ and $\cat A = \Om$.

To check the conditions fix some point $x \in X$ and some perfect dg $\underline \R$-module $G$. We consider the smooth homotopy equivalence $\set R\otimes \uEnd(G) \to \Om(U)\otimes \uEnd(G)$ given by inclusion and evaluation at $x$.
Then we apply Corollary \ref{cor:deRham}(i) to verify that the de Rham algebra satisfies (**) and the assumptions of Proposition \ref{prop-cohesivesheaves}.
\end{proof}

\begin{rk}\label{rk-infinitylocal}
To recover the results of \cite{Block09} we use Corollary \ref{cor-cohesiveperfect} and then recall that perfect dg sheaves over $\underline{\set R}$ are clc sheaves by Lemma \ref{lemma-perfectconstant}, which are in turn equivalent to various other notions of \emph{infinity local systems}.

In fact, under mild assumptions, the following are all quasi-equivalent dg categories:
\begin{enumerate}
\item perfect clc sheaves, sometimes called homotopy locally constant sheaves, i.e.\ fibrant cofibrant dg sheaves whose cohomology sheaves are locally constant of finite rank,
\item perfect dg modules over the dg algebra of chains on the Moore loop space of $X$,
\item the dg category obtained from the cotensor action of singular simplices on $X$ on the dg category of perfect chain complexes, see \cite{Holstein3},
\item (combinatorial) infinity local systems on a simplicial set as explicitly described in terms of a Maurer-Cartan condition in \cite{Block09}.
\end{enumerate}
One can extend all these notions by dropping the assumption of perfectness and the quasi-equivalences still hold.

The equivalence of (1) and (2) follows from \cite{Holstein1} and \cite{Holstein2}, (2) and (3) are identified in \cite{Holstein1}. The correspondence of (3) and (4) follows from \cite{Holstein3}; note that there is a difference of definition between the objects considered in (3) and (4) for an arbitrary simplicial set, but on fibrant simplicial sets the definitions agree.
In \cite{Holstein1} it is shown that all of these can be interpreted as categorified cohomology of $X$, i.e.\ cohomology of $X$ with coefficients in the dg category of perfect complexes.
Keeping with this viewpoint one could consider the dg category of cohesive modules over $\Om(X)$ as categorified de Rham cohomology.

Unravelling definitions we may also see that the category (4) for a reduced simplicial set $K$ agrees precisely with our definition of $\Tw(K)$. One may deduce that the two notions agree for arbitrary Kan complexes from Corollary \ref{cor-twistedreduced} and homotopy invariance of infinity local systems.

The main result of \cite{Block09} shows that if $X$ is a compact manifold and $k = \set R$ then the dg category of infinity local systems as in (4) is $A_{\oo}$-quasi-equivalent to the dg category of $\Z$-graded connections, using computations with iterated integrals. In Theorem \ref{thm-derham} we directly establish a quasi-equivalence of cohesive modules for the de Rham algebra with (1). There is, incidentally, also a direct proof comparing $\Z$-graded connections to (2), also using iterated integrals \cite{Arias16}.
\end{rk}

We now extend this result to the case where we replace the manifold $X$ by a simplicial complex $K$. We write $\Om(K)$ for de Rham algebra of piecewise smooth differential forms on $K$. Piecewise smooth differential forms define a sheaf on the underlying topological space $|K|$ of $K$ that we also denote by $\Om$.

\begin{thm}\label{thm-derhampiecewise}
Let $K$ be a connected finite dimensional simplicial complex. Then the functor $F: \Twfg(K)\to \underline{\set R}\mods$ sending $E$ to $U \mapsto E \otimes_{\Om(X)} \Om(U)$ 
is quasi-fully faithful and induces an equivalence $\Hot(\Twfg(K)) \cong \Dperf(|K|, \underline{\set R})$.
\end{thm}
\begin{proof}
First we show that piecewise smooth functions (and thus piecewise smooth forms) form a fine sheaf on $K$.

It is enough to construct, given two closed subsets $A$ and $B$ of $K$, a section $s$ of $\Om^0$ that is equal to 1 on $A$ and $0$ on $B$. We proceed by induction on the dimension of the simplex. So assume we have constructed the restriction of $s$ to $k$-simplices and denote it by $s'$. Consider a $(k+1)$-simplex $L$. We have to check that we can separate $L\cap A$ and $L\cap B$ by a function that restricts to $s'$ on the boundary. As $\Om^0$ is fine on $L$, we may choose a section $t$ that is equal to 1 on $L\cap A$ and 0 on $L\cap B$. Then on the boundary of $L$ we observe that the function $t - s'$ is 0 on the intersections of $A$ and $B$ with the boundary of $K$. We can easily find a smooth function $t'$ on $K$ that restricts to $t-s'$ and which has support disjoint from $A$ and $B$. Then we let $s = t - t'$. We can clearly do this for all $(k+1)$-simplices simultaneously as we did not change $s'$.

Now we need to check that $\Om$ satisfies condition (*) to deduce the theorem from Theorem \ref{thm-twistedsheaves2}. The other conditions on $(|K|,\Om)$ are immediate. Note that we cannot use Theorem \ref{thm-twistedsheaves} as $\Om(K)$ is not flat over $\Om^0(K)$.

Let $x \in |K|$. There is a neighbourhood $U$ of $x$ and a piecewise linear contracting homotopy $H: U \times [0,1] \to U$. This induces a map $H^{*}: \Om(U) = \lim_{\Delta \in K} \Omega(U \cap \Delta) \to \lim_{\Delta \in K} \Omega((U \cap \Delta) \times [0,1]) \cong  \Omega(U) \otimes \Omega[0,1]$.
Here for the last equivalence we use Corollary \ref{cor:product}.
Thus the map $H^{*}$ gives a smooth homotopy equivalence between $\set R$ and $\Omega(U)$.

Now we apply Theorem \ref{thm:homotopyequivalence1} to deduce condition (**) and apply Theorem \ref{thm-twistedsheaves2}.
\end{proof}

\begin{rk}
Note that Theorem \ref{thm-derhampiecewise} would be false for the polynomial de Rham algebra that is used for example in rational homotopy theory, cf.\ Example \ref{eg-polyderham}.
\end{rk}

\subsection{The Dolbeault algebra}\label{sect-dolbeault}

In this subsection $k$ is $\C$ and $X$ is a (not necessarily compact) complex manifold equipped with its sheaf of holomorphic functions $\cat O_{X}$. We revisit Block's proof \cite{Block10} that the derived category of perfect dg coherent sheaves on a complex manifold $X$ is equivalent to the homotopy category of cohesive modules over the Dolbeault algebra $(\cat A^{0*}(X), \bar\partial)$.
Note that the main result in the previous section draws from the methods in \cite{Block10};
we have generalised the setting and added some details regarding faithfulness of the functor from twisted modules to perfect dg sheaves.

Thanks to Lemma \ref{lemma-surjective} we may answer the implicit question in \cite[Remark 4.1.4]{Block10}.

\begin{thm}
	Let $X$ be a complex manifold. The functor $F: \Twfg(\cat A^{0*}(X), \bar \partial) \to \cat O_{X}\mods$ sending $V \otimes \cat A^{0*}(X)$ to its dg associated sheaf
	induces an equivalence $\Hot(\Twfg{(\cat A^{0*}(X), \bar\partial)}) \cong \Dperf^{B}(X, \cat O_{X})$.
\end{thm}
\begin{proof}
	The equivalence of homotopy categories follows from  Theorem \ref{thm-twistedsheaves} with $\cat R = 
	\cat O_X$ and $\cat A = (\cat A^{0*}(X), \bar\partial)$. Condition (*) is exactly the content of \cite[Lemma 4.1.5]{Block10}.
\end{proof}

\begin{rk}
One might try to also view this result through a suitable Schlessinger-Stasheff theorem.
There are, however, considerable conceptual obstacles to implementing this.
Note that the inclusion $\cat O(U) \to (\cat A^{0*}(U),\bar \partial)$ does not have a section as a function
 of topological vector spaces for any open set $U$, see  \cite[Proposition 5.4]{Mityagin71}.
 On a closed poly-disk $D$ there is a section (not compatible with restrictions), but it is not clear how to construct a homotopy equivalence between $\cat O(D)$ and $(\cat A^{0*}(D),\partial)$, or even what the correct notion of homotopy equivalence would be.
\end{rk}

\subsection{The singular cochain algebra}\label{sect-singular}

In this subsection $X$ is a topological space and $C^{*}(X)$ the pseudo-compact dg algebra of its normalized singular cochains. We will assume that $X$ is connected and locally contractible, and that
 $k$ has finite homological dimension.

We will consider infinitely generated modules, so recall from Section \ref{sect-notations} that whenever we consider $M\otimes C^*(X)$ for some graded $k$-module $M$ we will understand it as the completed tensor product.

To define a functor from $\Tw(C^*(X))$ to $\underline k\mods$ we recall that the presheaf of singular cochains with coefficients in any abelian group $L$, given for an open set $U\subset X$ by $U \mapsto C_{\text {sing}}^*(U, L)$, has a sheafification given by $U \mapsto  C^{*}_{\text{sing}}(U,L)/C_0^{*}(U,L)$. Here $C_{0}^{*}(U,L)$ consists of those singular cochains on $U$ such that there is an open cover of $U$ on which they vanish. See \cite{Sella16} for details.
We write $\cat C^{*}(L) = (\cat C^{*}(L), d_{\cat C})$ for the normalization of $C_{\text{sing}}^{*}(U,L)/C_{0}^{*}(U,L)$. This is a flabby sheaf if $X$ is semi-locally contractible and there is a quasi-isomorphism $\underline L \to \cat C^{*}(L)$. When $L = k$ we drop it from the notation, and we note that $\underline L \otimes \cat C^{*} \cong \cat C^{*}(L)$.

Let us consider the dg functor $F: \Tw(X) \to \underline k\mods$ defined by  
\begin{equation}\label{eq:def-F}{F(M)(U)=M\otimes_{C^{*}(X)}\cat C^*(U)},
\end{equation} where $U$ is an open subset of $X$. Note that as $C^{*}(X)$ is different from $\cat C^{*}(X)$ this differs from Definition \ref{defn-F}.
Then we have the following result.
\begin{thm}\label{thm-singularcochains}
The dg functor $F: \Tw(X) \to \underline k\mods$ defined above
 is quasi-fully faithful, and induces an equivalence $\Hot(\Tw(X)) \cong \Dlf(X, \underline k)$.
\end{thm}
The proof is somewhat long and technical and will occupy the rest of the paper. Many of the technical complications of the proof disappear under the assumption that $\ground$ is a field.

Given a simplex $\sigma$ we will denote its vertices by $\sigma_{0}, \dots, \sigma_{n}$ and write $\sigma_{i_{0}\dots i_{k}}$ for the subsimplex spanned by $\sigma_{i_{0}}, \dots, \sigma_{i_{k}}$.

Recall that we may write objects of $\Tw(X)$ as $(V \otimes C^{*}(X), D_{V})$ where $V$ is some free graded $\ground$-module.

\begin{lemma}\label{lemma-locallyconstant}
Let $V$ be a $\ground$-module considered as a dg module concentrated in degree 0. Then there is a bijective correspondence between $C^*(X)$-modules of  the form $(V \otimes C^*(X), D_V)$ and functors from the fundamental groupoid $\Pi(X)$ of $X$ to $\End(V)$, where the latter is viewed as a linear category with one object.
\end{lemma}
\begin{proof}
It suffices to identify $\MC(C^{*}(X, \End(V)))$ with functors from $\Pi(X)$ to $\End(V)$.
Consider an element $f\in \MC(C^{*}(X, \End(V)))$, which is by definition a 1-cochain. To define the functor $\Phi(f): \Pi(X) \to \End(V)$ it suffices to specify it on the morphisms of $\Pi(x)$. For a singular 1-simplex $\sigma$ of $X$, viewed as a morphism of $\Pi(X)$, set  $\Phi(f)([\sigma]) = 1+f(\sigma)$.

Assuming $f$ is MC we obtain, for any singular 2-simplex $\tau$,
\[ 
0 = (df + f^{2})(\tau) = f(\tau_{01}) - f(\tau_{02})+ f(\tau_{12}) + f(\tau_{01})f(\tau_{12}).
\]
We deduce that $1+f(\tau_{02}) = (1+f(\tau_{01}))(1+f(\tau_{12}))$. This shows that homotopic paths have the same image and concatenation of paths is sent to multiplication, so $\Phi(f)$ is a well-defined functor.

Conversely, given a functor $F: \Pi(X) \to \End(V)$, define a cochain $\Psi(F)$ by $\Psi(F)(\sigma) = F([\sigma]) - 1$. Given any two-simplex $\tau$ we know $F(\tau_{01}) \circ F(\tau_{12}) = F(\tau_{02})$, and the same computation as above shows $\Psi(F)$ is MC. The maps $\Psi$ and $\Phi$ are inverse to each other.
\end{proof}

Lemma \ref{lemma-locallyconstant} is compatible with the correspondence between locally constant sheaves and representations of the fundamental groupoid.
With the notation of the proof, for every MC element $x \in \MC(\End(V) \otimes C^{*}(X))$ we have that $(\underline V \otimes \cat C^{*}, d + x)$ is a soft resolution of the locally constant sheaf $\cat V$ associated to $\Phi(g)$. To check the monodromy we may observe that if $f: C_{*}(X) \to V$ represents a section and $\sigma$ is a 1-simplex connecting two points $\sigma_{1}$ and $\sigma_{0}$ then the cocycle condition $(d_{C}+g)(f)(\sigma) = 0$ gives $f(\sigma_{0}) = \Phi(g)([\sigma])(f(\sigma_{1}))$.

\begin{lemma}\label{lemma-mcsheafcohomology}
Let $A, B$ be $k$-algebras concentrated in degree zero and $x, y$ be MC elements in $A \otimes C^{*}(X)$ and $B \otimes C^{*}(X)$ respectively. Let $M$ be a $(B,A)$-bimodule and consider the dg module $(M \otimes C^{*}(X), D_{M})$ where $D_{M}(f) = df + yf  - (-1)^{|f|}fx$. Then the natural quotient map $q_{M}:M \otimes C^{*}(X) \to M\otimes \cat C^{*}(X)$ is a quasi-isomorphism.
\end{lemma}
\begin{proof}
We proceed exactly in the same way as to establish the quasi-isomorphism $C^{*}(X) \to \cat C^{*}(X)$. For the reader's convenience we provide some details. We first observe that $\cat C^{*}(X) = \varinjlim C^{*}_{\mathfrak U}(X)$ where the limit is over covers of $X$ and $C^{*}_{\mathfrak U}(X)$ are those singular cochains which vanish on $\mathfrak U$. So it suffices to show $q_{M}^{\mathfrak U}: M \otimes C^{*}(X)\to M\otimes C^{*}_{\mathfrak U}(X)$ is a quasi-isomorphism for every cover $\mathfrak U$. 

The natural quotient map $q^{\mathfrak U}: C^{*}(X)\to C^{*}_{\mathfrak U}(X)$ has a homotopy inverse $P$ induced by iterated barycentric subdivision of simplices. Inspection of the proof e.g.\ in \cite[Proposition 2.21]{Hatcher02}
shows that this homotopy equivalence is entirely formal, depending only on the boundary operator $d$ and the operator $b$ induced by taking the cone over the barycentre of a simplex. As long as they satisfy $db + bd = \id$ one may define a subdivision chain map $S$ and the homotopy $T$ from $S$ to the identity, and use these to define the desired map $P$, see below.
Thus we may repeat the whole construction with twisted coefficients.

Write $X = x+1$ and $Y = y+1$. 
We may write the differential on $M \otimes C^{*}(X)$ as 
\[
(D_{M}f)(\sigma) = Y(\sigma_{01})f(\partial_{0}\sigma) + \sum_{i=1}^{n-1}(-1)^{i} f(\partial_{i}\sigma) + (-1)^{n}f(\partial_{n})X(\sigma_{n-1,n}).
\]
We observe that this is just a two-sided version of the usual singular cochain complex with local coefficients.

Let $\beta(\sigma)$ denote the cone over the barycentre of $\sigma$. We then define $b_{M}f(\sigma)$ as $Y(\sigma_{0b})f(\beta\sigma)$ where $\sigma_{0b}$ denotes the 1-simplex connecting $\sigma_{0}$ to the barycentre of $\sigma$. A straightforward computation, using the fact that $Y(\sigma_{0b})Y(\sigma_{b0}) = \id$ by Lemma \ref{lemma-locallyconstant}, shows $b_{M}D_{M}+D_{M}b_{M} = \id$. 

We inductively define a twisted subdivision ${S_M}(f) = D_{M}S_{M}b_{M}(f)$, with $S_{M}(f) = f$ on a $0$-cochain, and a chain homotopy $T_{M}(f) = (\id - D_{M}T_{M})b_{M}$, with $T_{M}(f)= 0$ on a $0$-cochain. Then $D_{M}T_{M} + T_{M}D_{M} = \id - S_{M}$.
For $m \geq 0$ let $H_{m}= \sum_{0\leq i < m} S^{i}_{M}T_{M}$, this is a chain homotopy from $\id$ to $S_{M}^m$. For every simplex $\sigma$ there is a minimum $m(\sigma)$ such that $\beta^{m(\sigma)}(\sigma)$ is contained in $\mathfrak U$. We then define the map $H$ by $H(f)(\sigma) = H_{m(\sigma)}(f)(\sigma)$ and the map $P_{M} = S_{M}^{m(\sigma)} + D_{M}H_{m(\sigma)} - D_{M}H$.
One can check that $H$ is a chain homotopy between the identity and $P_{M}\circ q^{\mathfrak U}_{M}$. 
Moreover $P_{M}$ is a right inverse of $q^{\mathfrak U}_{M}$. This establishes the desired homotopy equivalence. Details are as for the untwisted dual case, which may be found in \cite{Hatcher02}.
\end{proof}
\begin{rk}
	Associated to a representation of the fundamental groupoid $R:\Pi(X)\to\End(V)$ is, according to Lemma \ref{lemma-locallyconstant}, an MC element $\Psi(R)$. The corresponding twisted module $(V\otimes C^*(X),D_V)\cong (V\otimes C^*(X))^{[\Psi(R)]}$ coincides with the singular complex of $X$ with local coefficients corresponding to the representation $R$. To a $\Pi(X)$-bimodule, i.e.
	a representation $P:\Pi\times\Pi^{\text{op}}\to\End(M)$ one similarly associates a pair $\Psi_1(P),\Psi_2(P)$ of MC elements and a two-sided twisted complex $(M\otimes C^*(X))^{[\Psi_1(P),\Psi_2(P)]}$; this complex was used in the proof of Lemma \ref{lemma-mcsheafcohomology}. Any $\Pi(X)$-bimodule determines, via the canonical functor $\Pi(X)\to\Pi(X)\times\Pi(X)^{\text{op}}$, a left $\Pi(X)$-module. It is easy to see that for a singular $n$-cochain $f$ with values in $M$, the map
	\[
	f\to (f)\cdot(\sigma_{n, n-1}\cdot \sigma_{n-1,n-2}\cdot\ldots\cdot\sigma_{1,0})
	\]
determines an isomorphism from the two-sided complex with local coefficients to the one-sided complex. This is analogous to the well-known isomorphism between the two-sided Hochschild complex of a group and a one-sided complex, cf. \cite[Chapter 6, p. 293]{MacLane67}.	
\end{rk}

The following lemma is only needed if $\ground$ is not a field. In that case not all locally constant sheaves are locally free, but the underlying graded module of a twisted module needs to be free.
\begin{lemma}\label{lemma-freeresolution}
Any locally constant sheaf $\cat V$ on $X$ is the image under $F$ of a bounded twisted $C^*(X)$-module $W\otimes C^*(X)$.
\end{lemma}
\begin{proof}
By Lemma \ref{lemma-locallyconstant} we know there is a $C^*(X)$-module $(V\otimes C^*(X), D_V)$ mapping to $\cat V$. The only problem is that $V$ might not be free over $k$. We pick a finite free resolution $q: (W, d_{W}) \to V$. Now we need to construct a differential $D_W$ on $W \otimes C^*(X)$ that maps to $D_V$.

$D_V$ is determined by the map $D^{1}_{V}|_{V}: V \to V \otimes C^{1}(X)$. For degree reasons all the maps $D^{i}_{V}: V \to V \otimes C^{i}(X)$ for $i \neq 1$ are zero. 

As $W$ is free we can lift $D_{V}^{1}$ to a chain map $w_{1}: W \to W \otimes C^{1}(X)$. As $D_{V}^{2}=0$ we know $w_{1}^{2} = d_{W}(w_{2})$ for some $w_{2}: W \to W[-1] \otimes C^{2}$.
We let $w_{0}= d_{W}$. Then this is the beginning of an inductive construction of a homomorphism $\sum_{i \geq 0} w_{i}$ that will define a differential on $W \otimes C^*(X)$. Assume we are given $w_{i}$ for $i \leq k$ satisfying $\sum_{i=0}^{n} w_{i}w_{n-i} = 0$ for every $n \leq k$. Then $\sum_{i=1}^{k} w_{i}w_{k+1-i}$ is an object of $\End(W)\otimes C^{k+1}$.

We now compute $[w_{0}, \sum_{i=1}^{k} w_{i}w_{k+1-i}]$ to check that $\sum w_{i}w_{k+1-i}$ is a $d_{W}$-cocycle. We observe that
\[
\sum_{0 \leq m, i, j \leq k; m+i+j=k+1} [w_{m}, w_{i}w_{j}] = 0
\]
by symmetry. Then we split the sum as
\[
\left[w_{0}, \sum_{i=1}^{k} w_{i}w_{k+1-i}\right] + \sum_{m = 1}^{k} \left[w_{m}, \sum_{i=0}^{k+1-m} w_{i} w_{k+1-m-i}\right] = 0
\]
But for $m \geq 1$ all $\sum_{i} w_{i} w_{k+1-m-i}$ are $0$ by induction. Thus the first term in the sum is $0$, which is what we had to show.

As $H^{-k}(\uEnd((W, d_{W}))) = \Ext^{-k}(V, V) = 0$ we see that the cocycle $\sum_{i=1}^{k} w_{i}w_{k+1-i}$ is a boundary and we can define $w_{k+1}$ such that $\sum_{i=0}^{k+1} w_{i}w_{k+1-i} = 0$. As $W$ is finite this process terminates. $D_{W}|_{W} \coloneqq \sum w_{i}$ defines a differential on $W\otimes C^*(X)$ that is compatible with $D_{V}$.

Now we filter $q: (W \otimes C^*(X), D_W) \to (V \otimes C^*(X), D_V)$ by the singular cochain degree. This is a complete exhaustive filtration and the associated graded map consists of quasi-isomorphisms $(W, d_{W})\otimes C^{p}(X) \simeq V \otimes C^{p}(X)$, thus $q$ is a quasi-isomorphism.

In fact, $q$ is a quasi-isomorphism if we replace $X$ by any open subset $U$ and thus we have constructed $(W \otimes C^*(X), D_{W})$ whose image under $F$ is quasi-isomorphic to $\cat V$.\end{proof}

\begin{lemma}\label{lemma-singularclc}
	For any twisted module $E$ the sheaf $F(E)$ is clc.
\end{lemma}
\begin{proof}
	Consider $P = (V \otimes C^*(X), D_V)$ in $\Tw(C^*(X))$ and restrict $P$ to a contractible subset $U \subset X$.
	We use the weak equivalence between $U$ and a point and apply Corollary \ref{cor:mctopologicalequivalence} to show that $P|_{U}$ is weakly equivalent to a constant sheaf on $U$ with fibre $(V, d_{V})$.
\end{proof}

\begin{lemma}\label{lemma-singularff}
	The natural functor $\Tw(C^{*}(X)) \to \cat C^{*}\mods$ is quasi-fully faithful.
\end{lemma}
\begin{proof}
	Given twisted modules $(V \otimes C^*(X), D_V)$ and $(W \otimes C^*(X), D_W)$ over $C^*(X)$ we know that $\uHom_{\Tw({C^*(X)})}(V \otimes C^{*}(X), W \otimes C^{*}(X))$ is given by $(\uHom(V,W)\otimes C^*(X))$
	with a differential defined by $f \mapsto D_W \circ f - (-1)^{|f|}(f \otimes \id_{\cat C^*(X)}) \circ D_V$ on $\uHom(V,W)$.
	
	We then compute $\uHom_{\cat C^{*}}(\underline V\otimes \cat C^*, \underline W \otimes \cat C^*)$. 
	There is an isomorphism of presheaves  between $U \mapsto 
	\uHom_{
		\mathcal C^*(U)}(\underline V\otimes \cat C^*(U), \underline W \otimes \cat C^*(U))$  and $U \mapsto \uHom(V, W) \otimes \mathcal C^*(U)$ with differenial induced by $f \mapsto D_{W}\circ f -(-1)^{|f|} \circ D_V$. 
	Thus the global sections of the sheafifications agree and we have
	$\uHom_{\cat C^{*}}(\underline V\otimes \cat C^*, \underline W \otimes \cat C^*) \cong \uHom(V, W) \otimes \cat C^{*}(X)$ with differential as above.
	
	Writing $M$ for $\uHom(V,W)$ we now need to show that there is a quasi-isomorphism $M \otimes C^{*}(X) \simeq M \otimes \cat C^{*}(X)$. Note that $M \otimes C^{*}(X)$ with its differential $D_{M}$ is not a $C^{*}(X)$-module, and in particular not a twisted module. We may still consider its sheafification.
	
	By Lemma \ref{lemma-singularclc} we know $\underline V\otimes \cat C^{*}$ and $\underline W \otimes \cat C^{*}$ are clc and thus so is $\sHom(V, W) \otimes \cat C^{*}$.
	Moreover, by Corollary \ref{cor-twistedreduced} we may assume that $V \otimes C^{*}$ and $W \otimes C^{*}$ are reduced, 
	so we may assume that $D^{0}_{M}$ induces a differential on $M$.
	
	We consider the natural map induced by the quotient $C^{*}(X) \to \cat C^{*}(X)$ and filter both 
	sides by the singular degree. We claim the associated spectral sequences agree on the second sheet, showing the map is a quasi-isomorphism as the filtration is complete exhaustive. For the first spectral sequence we have ${}^{I}E_{1}^{pq} = H^{q}(M\otimes C^{0}(X))\otimes_{C^{0}(X)} C^{p}(X)$, which we may rewrite as $H^{q}(M)\otimes C^{p}(X)$.
	The second sheet computes cohomology of a dg module $(H^{q}(M)\otimes C^*(X), d_{1})$, which satisfies the conditions of Lemma \ref{lemma-mcsheafcohomology} for $A = \End(V)$ and $B = \End(W)$.
	For the second spectral sequence one has ${}^{II}E_{1}^{pq}  = H^{q}(M)\otimes \cat C^{p}(X)$, and by Lemma \ref{lemma-mcsheafcohomology} the $E_{2}$-terms agree. 
	
	Here the first spectral sequence computes the Ext groups between clc sheaves using the singular cochain complex, and the second spectral sequence computes the Ext groups using a soft resolution. 
\end{proof}
Recall that to any dg  $\ground$-module $C^*$ one can associate its canonical truncation ${\tau_{\leq i}C^*}$ obtained by replacing 
$C^n$ with zero for $n>i$ and with 
${\ker(C^i\to C^{i+1})}$ for 
$n=i$. Then $\tau_{\leq i}C^*$ 
is a dg submodule in $C^*$ and we set $\tau_{\geq i+1}C^*:=C^*/(\tau_{\leq i}C^*)$. This construction works for complexes over any abelian category, in particular one can define canonical truncations for dg sheaves of $\ground$-modules. The following result shows that there are corresponding truncation functors for twisted $C^*(X)$-modules.
\begin{lemma}\label{lem:truncation}
For every twisted $C^*(X)$-module $M$ there is a twisted module $\tau_{\leq i}M$ and a map	$\tau_{\leq i}M\to M$ such that 
$F(\tau_{\leq i}M)\to F(M)$ is isomorphic in the derived category of dg $\ground$-sheaves on $X$ to the canonical map $\tau_{\leq i}F(M)\to F(M)$. Similarly there is twisted module $\tau_{\geq i}M$ and a map	$M\to\tau_{\geq i}M$ such that 
$F(M)\to F(\tau_{\geq i}M)$ is isomorphic in the derived category of dg $\ground$-sheaves on $X$ to the canonical map $F(M)\to \tau_{\geq i}F(M)$.
\end{lemma}
\begin{proof}
	We will prove the statement for the truncation $\tau_{\leq i}$; the other claim for $\tau_{\geq i}$ will follow by taking $\tau_{\geq i}M$ to be the cone of the map $\tau_{\geq i-1}M\rightarrow M$. Let $(V \otimes C^*(X), D_{V})$ be a twisted $C^*(X)$-module that will be assumed to be reduced (or we replace it by a reduced one by Corollary \ref{cor-twistedreduced}). Note that $D_V$ restricts to $(\tau_{\leq i}(V) \otimes [C^*(X)])$ and so  
	$ (\tau_{\leq i}(V) \otimes [C^*(X)], D_{V})$ is well-defined as a dg $C^*(X)$-module. This may not be a twisted $C^*(X)$-module since $\tau_{\leq i}(V)$ may not be free over $\ground$. 
	
	We pick a $\ground$-free resolution $(W,d_W)$ of $\tau_{\leq i}(N)$ and, arguing as in the proof of Lemma \ref{lemma-freeresolution}, construct a differential $D_W$ on $W\otimes C^*(X)$ together with a filtered quasi-isomorphism
	$(W\otimes C^*(X), D_W)\to (\tau_{\leq i}(V) \otimes [C^*(X)], D_{V})$.
	
	Let us set $\tau_{\leq i}(V \otimes C^*(X), D_{V}):=(W\otimes C^*(X), D_W)$. We need to show that the truncation so obtained agrees with the truncation of dg sheaves upon applying the functor $F$. This is a local statement, and so it suffices to prove it with $X$ replaced by a small contractible neighbourhood $U\subset X$.  This is, however, obvious since the twisted $C^*(U)$-module $(W\otimes C^*(U), D_W)$
	is homotopy equivalent to the (untwisted) tensor product of complexes $(W,d_W)$ and $C^*(U)$ by Corollary \ref{cor:mctopologicalequivalence}.
\end{proof}
\begin{proof}[Proof of Theorem \ref{thm-singularcochains}]
We use Lemma \ref{lemma-perfectconstant} to identify $\Dlf(X, \underline k)$ with cohomologically constant sheaves. Then by Lemma \ref{lemma-singularclc} the image of $F$ consists of locally free dg sheaves. Next we show that the functor $\Hot(\Tw(X)) \to \Dlf(X, \underline k)$ induced by $F$ is fully faithful.

To this end note that this functor can be represented as the following composition:
\[
\Hot(\Tw(X)) \to \Hot(\cat C^{*}\mods) \to \Dlf(X,  {\cat C^{*}}) \to \Dlf(X, \underline k)
\]
  By Lemma \ref{lemma-singularff} the first functor is fully faithful.
 To show the second functor is fully faithful on the image of $\Hot(\Tw(X))$ we claim $\uHom(\underline V \otimes \cat C^{*}, \underline W \otimes \cat C^{*}) \simeq R\uHom(\underline V \otimes \cat C^{*}, \underline W \otimes \cat C^{*})$.
We deduce this claim by following verbatim the proof of Lemma \ref{lemma-fullyfaithful}.
By Corollary \ref{cor-twistedreduced} (1)
 we have a homotopy equivalence $(V \otimes C^{*}(U), D_{W}) \simeq (V,d_{V})\otimes C^{*}(U)$ on any contractible set $U$. This takes the place of condition (*).
 We allow for unbounded dg sheaves, but this does not affect the proof as the filtration by degree of $\sHom(V,W)\otimes \cat C^{*}$ is still exhaustive and Hausdorff. Note that the dg $k$-module $(V, d_{V})$ is cofibrant as it is free in each degree and $k$ has finite homological dimension. Hence the associated sheaf $\underline V \otimes \cat C^{*}$ is also cofibrant.

Since $\underline \ground \simeq \cat C^{*}$ we have $D(X, \cat C^{*}) \cong D(X, \underline \ground)$ and $\Hot(\Tw(X)) \to D_{lf}(X, \underline \ground)$ is fully faithful. 
  
Moreover, as $\cat  C^{*}$ is projective over $k$ we may refine the argument and show, as in the proof of Corollary \ref{cor-quasifullyfaithful}, that the functor $F: \Tw(X) \to \underline \ground \mods$ is quasi-fully faithful. 

Now we determine the quasi-essential image of $F$.
The subcategory of $\Dlf(X, \underline k)$ given by bounded dg sheaves is the smallest triangulated subcategory inside the derived category of dg $\ground$-module sheaves on $X$ containing all locally constant sheaves. This follows since any bounded element in $\Dlf(X, \underline k)$ is an iterated extension of its cohomology sheaves.
The image of $F$ contains all locally constant sheaves by Lemma \ref{lemma-freeresolution}. Thus, since $F$ is compatible with cones, the quasi-essential image of $F$ contains all \emph{bounded} clc sheaves. 

 Observe that every bounded below clc sheaf $\cat M$ in $\Dlf(\underline k)$ is a homotopy colimit (in the sense of \cite{Bokstedt93}) of its truncations, $\hocolim_i \tau_{\leq i} \cat M \simeq \cat M$.
By Lemma \ref{lemma-freeresolution} we may lift the diagram $\tau_{\leq i}\cat M$ to a diagram $\{P_i\}$ in $\Hot(\Tw(X))$. As $\Hot(\Tw(X))$ has arbitrary direct sums we may define $P = \hocolim_i P_{i}$ and there is a natural map $\cat M \to F(P)$ which is an isomorphism on cohomology (as we can check on truncations using $\tau_{\leq i} F(P) \simeq F(\tau_{\leq i} (P)) \simeq \tau_{\leq i} \cat M$ by Lemma \ref{lem:truncation}). Thus, all \emph{bounded below} clc sheaves are in the quasi-essential image of $F$.

Finally we write a bounded above clc sheaf $\cat M$ as a limit of bounded dg sheaves; $\cat M \cong \lim \tau_{\geq i} \cat M$.
We will explicitly construct a twisted $C^*(X)$-module $Q$ with a map $F(Q) \to \cat M$ such that $\tau_{\geq i}F(Q)  
\simeq \tau_{\geq i}\cat M$, showing $F(Q) \simeq \cat M$.

To find $Q$ we proceed as follows. We fix some $Q_{0} = Q_{0}' \otimes C^{*}(X)$ mapping to $\tau_{\geq 0}\cat M$ and then construct twisted modules $Q_{i} = Q_{i}'\otimes C^{*}(X)$, $i < 0,$ inductively.  
We may write $\tau_{\geq i} \cat M$ as an extension of $\tau_{\geq i+1}\cat M$ by $H^{i}(\cat M)[-i]$.
Using Lemma \ref{lemma-freeresolution} we obtain  $W_{i} \otimes C^*(X)$ mapping to $H^{i}(\cat M)[-i]$ under $F$ where $W_{i}$ is a finite complex of free $\ground$-modules; moreover, because $\ground$ is of finite homological dimension $\operatorname{gd}(\ground)$, the length of $W_i$ is bounded independently of $i$.
 
As $F$ is quasi-full we may lift the extension map $\tau_{\geq i+1}\cat M \to H^{i}(\cat M)[-i+1]$ to $\eta_{i}: Q_{i+1} \to W_{i}\otimes C^*(X)[1]$. Now the cone on $\eta_{i}$ is defined as the twisted module of the form $Q'_{i} \otimes C^{*}(X)$ where $Q'_{i} = Q'_{i+1} \oplus W_{i}[1]$ and the differential is given by $D_{Q_{i}} = (D_{Q_{i+1}}, D_{W} + \eta_{i})$, see Section \ref{twisted}.
Thus we let $Q_{i}$ be the cone of the map $\eta_{i}$. By construction there is a quasi-isomorphism $F(Q_{i}) \simeq \tau_{\geq i}\cat M$. 

By construction $Q_{i}'$ is eventually constant, to be precise the degree $m$ part $(Q_{i}')^{m}$ is independent of $i$ if $i < m - \operatorname{gd}(\ground) -1$. We define a graded $\ground$-module $Q'$ by  $(Q')^{m} \coloneqq (Q_{i}')^{m}$ for sufficiently small $i$.
Similarly, the differential $D_{Q_{i}}$ restricted to $Q_{i}'$ is eventually constant and we define $D_{Q}$ on $(Q')^{m}$ to be $D_{Q_{i}}$ (for sufficiently small $i$) and extend by the Leibniz rule.

Then $Q = (Q' \otimes C^{*}(X), D_{Q})$ is the desired twisted module. There is a natural map $Q \to Q_{i}$ and the maps $F(Q_{i}) \to \tau_{\geq i}\cat M$ induce a map $F(Q) \to \cat M$.
We need to check that $\tau_{\geq i}F(Q) \simeq F(\tau_{\geq i}Q)$ is equivalent to $\tau_{\geq i}\cat M$. By construction $\tau_{\geq i+1} Q'  = Q'_{i+1} \oplus \tau_{\geq i+1}W_{i}[1] \oplus \tau_{\geq i+1}W_{i-1}[2] \oplus \dots$. All summands but the first are acyclic for $D^{0}$, thus after applying $F$ we can show that $F(\tau_{\geq i} Q) \simeq \tau_{\geq i}\cat M$.
 
This shows that that every \emph{bounded above} clc sheaf is in the quasi-essential image of $F$. 

As every dg sheaf $\mathcal F$ is an extension of a bounded above sheaf $\tau_{\geq 0}\mathcal F$ by a bounded below dg sheaf $\tau_{\leq 0}\mathcal F$, it follows that $F$ is quasi-essentially surjective.
\end{proof}
\begin{cor}
With $X$ as above $F$ induces an equivalence $\Hot(\Twfg(X) ) \to \Dperf(X, \underline k)$.
\end{cor}
\begin{proof}
We follow the proof of Theorem \ref{thm-singularcochains}. In particular this means we define the functor $F$ on $\Twfgfree(X)$, the dg category of finitely generated twisted modules, and obtain an embedding $\Hot(\Twfgfree(X)) \to \Dperf(X, \underline k)$. As the right hand side is idempotent complete this extends to an embedding $\Hot(\Twfg(X)) \to \Dperf(X, \underline k)$.

Essential surjectivity needs a little extra care. Considering any perfect $\underline\ground$-module as a finite extension of its cohomology sheaves it suffices to find a preimage for a locally constant sheaf $\cat M$. 

The fiber of $\cat M$ may not be free, but by Lemma \ref{lemma-perfectconstant} it is quasi-isomorphic to a strictly perfect dg module $P$ over $k$. 
Next we find a dg module $Q$ over $k$ such that $P \oplus Q$ is free and of finite rank in each degree, and $Q$ has cohomology concentrated in degree $0$.
To do this let us write $P$ as $P^{n}\to \dots \to P^{0}$. We pick for each $P^{i}$ a $k$-module $R^{i}$ such that $P^{i}\oplus R^{i}$ is free of finite rank. Then let $Q^{i} = R^{i} \oplus \bigoplus_{j = i+1}^{n} R^{j}\oplus P^{j}$ and define differentials inductively. The map $d^{n}: Q^{n}\to Q^{n-1}$ is just the inclusion of $R^{n}$, and $d^{i}$ is defined as the natural inclusion into $Q^{i-1}$ of the cokernel of $d^{i-1}$. With this definition the cohomology of $Q$ is a $k$-module $N$ concentrated in degree 0. 

We now consider the locally constant sheaf $\cat M' = \cat M \oplus \underline N$. By construction its fiber has a finite free resolution of finite rank. We use Lemma \ref{lemma-freeresolution} to lift $\cat M'$ to a finitely generated twisted module, using the fact that we may choose $W$ in the proof of Lemma \ref{lemma-freeresolution} to be of finite rank.  But  $\cat M$ is a summand of $\cat M'$. Thus, as $\Twfg(X)$ is equivalent to an idempotent complete subcategory of $\Dperf(X, \underline k)$, it follows that $\cat M$ is in the essential image of $\Twfg(X)$.
\end{proof}

\begin{rk} 
For an early incarnation of MC elements on singular cochains see \cite{BrownJr59}. There twisting cochains are used to express singular chains on a fibre space in terms of singular chains on base and fibre. One may interpret this as  higher MC elements on $C^{*}(X)$ representing certain infinity local systems.
\end{rk}

There is a version of Theorem \ref{thm-singularcochains} for simplicial sets.

\begin{cor}
Let $X$ be a connected Kan complex. Then there is a quasi-fully faithful functor from $\Tw(X)$ to the category of dg sheaves of $\ground$-modules on $|X|$, the geometric realization of $X$, which induces an equivalence $\Hot(\Tw(X)) \cong \Dlf(|X|)$.
\end{cor}
\begin{proof}
The singular simplicial set of $|X|$, is weakly equivalent to $X$. Since both are Kan simplicial sets, by Corollary \ref{cor:weaklyKan} their categories of twisted modules are quasi-equivalent.
Now the result follows from \ref{thm-singularcochains} since $|X|$ is locally contractible.
\end{proof}

\appendix
\section{Nuclear spaces}\label{app-nuclear}

In this Appendix we collect some facts about Grothendieck's nuclear spaces used in the main text, for the reader's convenience. Our main 
sources are \cite{Jarchow12} and \cite{Treves67}. We will consider complete locally convex Hausdorff topological vector spaces over $\mathbb{R}$ which will be referred to below as simply `vector spaces'.  If we have a linear continuous injection $U\hookrightarrow V$ that is a homeomorphism of $U$ on its image, we will refer to $U$ as a subspace of $V$.

\begin{defn}Let $U$ and $V$ be  vector spaces. Their projective tensor product ${U\otimes_\pi V}$ is a vector space having a universal property with respect to continuous bilinear maps out of $U\times V$, i.e.\ for any vector space $W$ there is a natural isomorphism between the set of continuous bilinear maps $U\otimes_\pi V\to W$ and the space of bilinear continuous maps $U\times V\to W$.	
\end{defn}
It is clear that the above defines $U\otimes_\pi V$ up to a unique isomorphism, and there is an explicit construction (that we will not need) showing that the vector space with the required universal property exists. There are other natural notions of a tensor product of vector spaces, of which the most important is that of an \emph{injective} tensor product denoted by $U\otimes_\epsilon V$, \cite[Chapter 16]{Jarchow12}. There is a canonical continuous map $U\otimes_\pi V\to U\otimes _\epsilon V$.
\begin{defn}
	A vector space $U$ is called \emph{nuclear} if for any vector space $V$ the canonical map 
	$U\otimes_\pi V\to U\otimes_\epsilon V$ is a topological isomorphism.
\end{defn}
From now on we will refer to projective tensor products as simply tensor products and omit the corresponding subscript. 

The category of nuclear spaces and continuous linear maps is closed with respect to various natural operations.
\begin{thm}\label{thm:stable}
	The collection of nuclear spaces is stable with respect to forming arbitrary direct products, tensor products and passage to subspaces.
\end{thm}
\begin{proof}
	See \cite[Corollary 21.2.3]{Jarchow12}.	
\end{proof}	
\begin{cor}\label{cor:stable}
	The category of nuclear spaces contains arbitrary limits.
\end{cor}
\begin{proof}
	This follows immediately from Theorem \ref{thm:stable} since any limit can be constructed using direct products and passing to subspaces.
\end{proof}
It turns out that the operation of tensor product with a nuclear space commutes with arbitrary limits:
\begin{thm}\label{thm:limit}
	Let $U_\alpha$ be a diagram of vector spaces and continuous linear maps and $V$ be a nuclear space. Then there is a natural topological isomorphism
\[(\underleftarrow{\lim}_\alpha U_\alpha)\otimes V\cong \underleftarrow{\lim}_\alpha( U_\alpha\otimes V).\]
\end{thm}	
\begin{proof}
	It suffices to show that tensor products commutes with direct products and passing to kernels. This follows from  \cite[Proposition 16.2.5 and Theorem 16.3.1]{Jarchow12}, taking into account that injective and projective tensor products with a nuclear space are isomorphic.
\end{proof}	

A lot of vector spaces one encounters in analysis are nuclear. In particular:
\begin{thm}\label{thm:smooth}
	Let $W$ be an open subset of some topological  simplex $\Delta^{n}$. Then the algebra $C^\infty(W)$ of smooth functions on $W$ is 
	nuclear.
\end{thm}
\begin{proof}  Let $I^n_\epsilon\subset\mathbb{R}^n$ denote the $n$-dimensional cube in $\mathbb{R}^n$ with side of length $\epsilon>0$. Using Seeley's extension theorem, more specifically its version for domains with corners \cite[Proposition 24.10]{Kriegl97}, we conclude that the restriction map ${C^\infty(\mathbb{R}^n)\to C^\infty(I^n_\epsilon)}$ has a continuous splitting and, since $C^\infty(\mathbb{R}^n)$ is nuclear, \cite[Corollary to Theorem 51.5]{Treves67}, its retract $C^\infty(I^n_\epsilon)$ is likewise nuclear. Moreover, clearly the algebra of smooth functions on any closed subset in $\mathbb{R}^n$ diffeomorphic to $I^n_\epsilon$ also forms a nuclear space as it is isomorphic to $C^\infty(I^n_\epsilon)$.  We then deduce nuclearity of $W$ by representing it as a union of a collection of subsets diffeomorphic to $I^n_\epsilon$ and using Theorem \ref{thm:limit}. 
\end{proof}

Given a smooth manifold $X$ we consider its de Rham algebra $\Om(X)$ and for a simplicial complex $K$ we consider its piecewise smooth de Rham algebra $\Om(K)$. We also consider the piecewise smooth de Rham algebra on any open subset $U$ of the underlying space $|K|$ of $K$. Theorem \ref{thm:smooth} implies the following.
\begin{cor}\label{cor:smooth}
If $W$ be an open subset of $\mathbb{R}^n$ or of some standard simplex $\Delta^{n}$ then the dg algebra $\Omega(W)$ is
nuclear.	\qed
\end{cor}	

\begin{thm}\label{thm:product}
	Let $U, W$ be open subsets of topological simplices $\Delta^{n}$ and $\Delta^m$ respectively for some $n,m>0$. Then there is a natural topological isomorphism $\Omega(U\times W)\cong \Omega(U)\otimes \Omega(W)$.\
\end{thm}
\begin{proof}
It suffices to prove  the isomorphism $C^\infty(U\times W)\cong C^\infty(U)\otimes C^\infty(W)$. Arguing as in the proof of Theorem \ref{thm:smooth}, we represent $U$ and $V$ as unions of subsets diffeomorphic to cubes $I^n_\epsilon$ and $I^m_{\epsilon^\prime}$; it will be sufficient to prove the desired isomorphism for $U=C^\infty(I^n_\epsilon)$, $V=C^\infty(I^m_{\epsilon^\prime})$. 
Since $C^\infty(I^n_\epsilon)$ and $C^\infty(I^m_{\epsilon^{\prime}})$ are retracts of $C^\infty(\mathbb{R}^n)$ and $C^\infty(\mathbb{R}^m)$ respectively, the natural map ${C^\infty(I^n_\epsilon)\otimes C^\infty(I^m_{\epsilon^\prime})\to C^\infty(I^n_\epsilon\times I^m_{\epsilon^\prime})}$ is a retract of the map $C^\infty(\mathbb{R}^n)\otimes C^\infty(\mathbb{R}^m)\to C^\infty(\mathbb{R}^n\times \mathbb{R}^m)$ and so is an isomorphism since the latter map is, \cite[Theorem 51.6]{Treves67}.
\end{proof}

\begin{cor}\label{cor:product}
Let $U$ and $W$ be open subsets of the underlying spaces of simplicial complexes $K$ and $L$. Then $\Om(U \times W) \cong \Om(U)\otimes \Om(W)$.
\end{cor}
\begin{proof}
By definition $\Om(U) = \lim_{\Delta \in K} \Om(U \cap |\Delta|)$. As the tensor product commutes past the limits by Theorems \ref{thm:smooth} and \ref{thm:limit} it suffices to check the result for open subsets of the standard simplex, which is the content of Theorem \ref{thm:product}.
\end{proof}
\bibliography{./biblibrary2}

\end{document}